\documentclass[10pt]{article}

\usepackage{comment,url,algorithm,algorithmic,graphicx,subcaption,relsize}
\usepackage{amssymb,amsfonts,amsmath,amsthm,amscd,dsfont,mathrsfs,mathtools,multirow,microtype,nicefrac,pifont}
\usepackage{float,psfrag,epsfig,color,url,hyperref}
\usepackage{upgreek}
\usepackage[dvipsnames]{xcolor}
\usepackage{epstopdf,bbm,mathtools,enumitem}
\usepackage[toc,page]{appendix}
\usepackage[mathscr]{euscript}
\usepackage[giveninits=true, maxnames=10, sorting=nyt, style=alphabetic]{biblatex}

\usepackage[top=1in, bottom=1in, left=1in, right=1in]{geometry}

\def\balign#1\ealign{\begin{align}#1\end{align}}
\def\baligns#1\ealigns{\begin{align*}#1\end{align*}}
\def\balignat#1\ealign{\begin{alignat}#1\end{alignat}}
\def\balignats#1\ealigns{\begin{alignat*}#1\end{alignat*}}
\def\bitemize#1\eitemize{\begin{itemize}#1\end{itemize}}
\def\benumerate#1\eenumerate{\begin{enumerate}#1\end{enumerate}}

\newenvironment{talign*}
 {\csname align*\endcsname}
 {\endalign}
\newenvironment{talign}
 {\csname align\endcsname}
 {\endalign}

\def\balignst#1\ealignst{\begin{talign*}#1\end{talign*}}
\def\balignt#1\ealignt{\begin{talign}#1\end{talign}}

\let\originalleft\left
\let\originalright\right
\renewcommand{\left}{\mathopen{}\mathclose\bgroup\originalleft}
\renewcommand{\right}{\aftergroup\egroup\originalright}

\def\tinycitep*#1{{\tiny\citep*{#1}}}
\def\tinycitealt*#1{{\tiny\citealt*{#1}}}
\def\tinycite*#1{{\tiny\cite*{#1}}}
\def\smallcitep*#1{{\scriptsize\citep*{#1}}}
\def\smallcitealt*#1{{\scriptsize\citealt*{#1}}}
\def\smallcite*#1{{\scriptsize\cite*{#1}}}

\def\<{\left\langle} 
\def\>{\right\rangle}

\def\defeq{\coloneqq} 
\DeclareSymbolFont{rsfs}{U}{rsfs}{m}{n}
\DeclareSymbolFontAlphabet{\mathscrsfs}{rsfs}

\providecommand{\tr}{\mathop\mathrm{tr}}

\ifdefined\nonewproofenvironments\else
\ifdefined\ispres\else
\newtheorem{theorem}{Theorem}
\newtheorem{lemma}[theorem]{Lemma}
\newtheorem{corollary}[theorem]{Corollary}
\newtheorem{definition}[theorem]{Definition}

\renewenvironment{proof}{\noindent\textbf{Proof.}\hspace*{.3em}}{\qed \vspace{.1in}}
\newenvironment{proof-sketch}{\noindent\textbf{Proof Sketch}
  \hspace*{1em}}{\qed\bigskip\\}
\newenvironment{proof-idea}{\noindent\textbf{Proof Idea}
  \hspace*{1em}}{\qed\bigskip\\}
\newenvironment{proof-of-lemma}[1][{}]{\noindent\textbf{Proof of Lemma {#1}}
  \hspace*{1em}}{\qed\\}
  \newenvironment{proof-of-proposition}[1][{}]{\noindent\textbf{Proof of Proposition {#1}}
  \hspace*{1em}}{\qed\\}
\newenvironment{proof-of-theorem}[1][{}]{\noindent\textbf{Proof of Theorem {#1}}
  \hspace*{1em}}{\qed\\}
\newenvironment{proof-attempt}{\noindent\textbf{Proof Attempt}
  \hspace*{1em}}{\qed\bigskip\\}
\newenvironment{proofof}[1]{\noindent\textbf{Proof of {#1}}
  \hspace*{1em}}{\qed\bigskip\\}
 
\newtheorem*{remark*}{Remark}
\newenvironment{remark}{\noindent\textbf{Remark.}
  \hspace*{0em}}{\smallskip}

\fi

\newtheorem{assumption}{Assumption}
\theoremstyle{definition}
\newtheorem{example}[theorem]{Example}

\fi
\makeatletter
\@addtoreset{equation}{section}
\makeatother

\hypersetup{
  colorlinks,
  linkcolor={red!50!black},
  citecolor={blue!50!black},
  urlcolor={blue!80!black}
}

\definecolor{OliveGreen}{rgb}{0,0.6,0}

\definecolor{OliveGreen}{rgb}{0,0.6,0}
\usepackage{custom}
\usepackage[normalem]{ulem}

\allowdisplaybreaks{}
\makeatletter
\renewcommand{\paragraph}{%
  \@startsection{paragraph}{4}%
  {\z@}{1.25ex \@plus 1ex \@minus .2ex}{-1em}%
  {\normalfont\normalsize\bfseries}%
}
\makeatother

\title{Sampling from the Mean-Field Stationary Distribution}

\begin{document}

 \author{\!\!\!\!\!
Yunbum Kook\thanks{
  School of Computer Science at
  Georgia Institute of Technology, \texttt{yb.kook@gatech.edu}
} \ \ \
Matthew S.\ Zhang\thanks{
  Department of Computer Science at
  University of Toronto, and Vector Institute, \texttt{matthew.zhang@mail.utoronto.ca}
} \\
 Sinho Chewi\thanks{
  School of Mathematics at
  Institute for Advanced Study, \texttt{schewi@ias.edu}
 }
 \ \ \
 Murat A.\ Erdogdu\thanks{
  Department of Computer Science at
  University of Toronto, and Vector Institute, \texttt{erdogdu@cs.toronto.edu}
 }
 \ \ \
 Mufan (Bill) Li\thanks{
  Department of Operations Research and Financial Engineering at Princeton University \texttt{mufan.li@princeton.edu}
}
}

\maketitle

\begin{abstract}
    We study the complexity of sampling from the stationary distribution of a mean-field SDE\@, or equivalently, the complexity of minimizing a functional over the space of probability measures which includes an interaction term.
    Our main insight is to \emph{decouple} the two key aspects of this problem: (1) approximation of the mean-field SDE via a finite-particle system, via uniform-in-time propagation of chaos, and (2) sampling from the finite-particle stationary distribution, via standard log-concave samplers.
    Our approach is conceptually simpler and its flexibility allows for incorporating the state-of-the-art for both algorithms and theory. 
    This leads to improved guarantees in numerous settings, including 
    better guarantees for optimizing certain two-layer neural networks in the mean-field regime.
    A key technical contribution is to establish a new uniform-in-$N$ log-Sobolev inequality for the stationary distribution of the mean-field Langevin dynamics.
\end{abstract}

\section{Introduction}\label{scn:intro}
The minimization of energy functionals $\mc E$ over the Wasserstein space $\Pac$ of probability measures has attracted substantial research activity in recent years, encompassing numerous application domains, including distributionally robust optimization~\cite{Kuhn+19DRO, YueKuhWie22OptWass}, sampling~\cite{jordan1998variational, Wib18SamplingOpt, ChewiBook}, and variational inference~\cite{LiuWan16SVGD, Lam+22GVI, Diao+23FBGVI, JiaChePoo23MFVI, Lac23MFVI, YaoYan23MFVI}.

 A canonical example of such a functional is $\mc E(\mu) = \int V\,\D\mu + \int \log \mu \, \D \mu$, where $V : \R^d\to\R$ is called the potential.
 Up to an additive constant, which is irrelevant for the optimization, this energy functional equals the KL divergence $\KL(\mu\mmid \pi)$ with respect to the density $\pi\propto\exp(-V)$, and the celebrated result of~\cite{jordan1998variational} identifies the Wasserstein gradient flow of $\mc E$ with the Langevin diffusion.
 This link has inspired a well-developed theory for log-concave sampling, with applications to Bayesian inference and randomized algorithms; see~\cite{ChewiBook} for an exposition.

The energy functional above contains two terms, corresponding to two of the fundamental examples of functionals considered in Villani's well-known treatise on optimal transport~\cite{villani2003topics}.
Namely, they are the ``potential energy'' and the entropy, the latter being a special case of the ``internal energy.''
However, Villani identifies a third fundamental functional---the ``interaction energy''---with the \emph{pairwise} form given by
\begin{align}\label{eq:total_energy_with_interaction}\tag{$\msf{pE}$}
    \mc{E}(\mu) \deq \int V(x) \, \mu(\D x) + \iint W(x - y) \, \mu(\D x)\, \mu(\D y) + \frac{\sigma^2}{2} \int \log \mu(x) \, \mu(\D x)\,.
\end{align}
More generally, in this work we consider minimizing the \emph{generic} entropy-regularized energy
\begin{align}\label{eq:total_energy}\tag{$\msf{gE}$}
    \mc E(\mu)  \deq  \mc F(\mu) + \frac{\sigma^2}{2} \int \log \mu \, \D \mu \,
\end{align}
where $\mc F : \Pac \to \R$ is a known functional. The minimization of the energy~\eqref{eq:total_energy} has recently been of interest due to its role in analysing neural network training dynamics in the mean-field regime, including with~\cite{suzuki2022uniform} and without~\cite{chizat2018global, mei2018mean} entropic regularization, as well as with Fisher regularization~\cite{Cla+23MeanFieldFI}.

For the sake of exposition, let us first focus on minimizing the pairwise energy~\eqref{eq:total_energy_with_interaction}.
A priori, this question is more difficult than log-concave sampling; for instance, $\pi$ does not admit a closed form but rather is the solution to a non-linear equation 
\begin{align}\label{eq:mean_field_stat_interactive}
    \pi(x) \propto \exp \Bigl(-\frac{2}{\sigma^2}\, V(x) - \frac{2}{\sigma^2} \int W(x-\cdot) \, \D\pi\Bigr)\,.
\end{align}
However, here too there is a well-developed mathematical theory which suggests a principled algorithmic approach.
Just as the Wasserstein gradient flow of~\eqref{eq:total_energy_with_interaction} in the case when $W = 0$ can be identified with the Langevin diffusion, the Wasserstein gradient flow of~\eqref{eq:total_energy_with_interaction} in the case when $W \ne 0$ corresponds to a (pairwise) McKean{--}Vlasov SDE, i.e. an SDE whose coefficients depend on the marginal law of the process, given below as
\begin{align}\label{eq:mean_field_interactive} \tag{$\msf{pMV}$}
    \D X_t = -\Bigl(\nabla V(X_t) + \int \nabla W(X_t - \cdot) \, \D \pi_t \Bigr)\, \D t + \sigma \, \D B_t\,,
\end{align}
where $\pi_t = \law(X_t)$,  $W$ is even, and $\{B_t\}_{t\ge 0}$ is a standard Brownian motion on $\R^d$.
Since the McKean{--}Vlasov SDE is the so-called \emph{mean-field limit} of interacting particle systems, we can approximately sample from the minimizer $\pi$ by numerically discretizing a system of SDEs, which describe the evolution of $N$ particles$\{X_t^{1:N}\}_{t\ge 0}  \deq  \{(X^1_t,\dotsc, X^N_t)\}_{t \geq 0}$ as:
\begin{align}\label{eq:finite_particle_interactive} \tag{$\msf{pMV}_N$}
    \D X_{t}^{i} = -\Bigl(\nabla V(X_{t}^{i}) + \frac{1}{N-1}\sum_{j \in [N]\backslash i} \nabla W(X_{t}^{i} -  X_{t}^{j}) \Bigr)\, \D t + \sigma \, \D B_{t}^i\,, \quad \forall\, i \in [N]\,,
\end{align}
where $\{B^i : i\in [N]\}$ is a collection of independent Brownian motions.
Moreover, the error from approximating the mean-field limit via this finite particle system has been studied in the literature on \emph{propagation of chaos} \cite{sznitman1991topics}.
Similarly, the Wasserstein gradient flow for~\eqref{eq:total_energy} corresponds to the \emph{mean-field Langevin dynamics} and admits an analogous particle approximation.

The bounds for propagation of chaos have been refined over time, with \cite{lacker2023sharp} recently establishing a tight error dependence $\mathcal{O}(\nicefrac 1 N)$ on the total number of particles $N$. 
These bounds, however, do not translate immediately into algorithmic guarantees. 
Existing sampling analyses study the propagation of chaos and discretization as a single \textbf{entangled} problem, which thus far have only been able to use weaker ${\mc O}(\sqrt{\nicefrac 1 N})$ rates for the former. 
Furthermore, there has been recent interest in using more sophisticated particle-based algorithms, e.g., ``non-linear'' Hamiltonian Monte Carlo~\cite{bou2023nonlinear} and the mean-field underdamped Langevin dynamics~\cite{fu2023mean} to reduce the discretization error. 
Currently, this requires repeatedly carrying out the propagation of chaos and time discretization analyses from the ground up for each instance. 

This motivates us to pose the following questions: \textbf{(1)} Can we incorporate improvements in the propagation of chaos literature, such as the $\mc{O}(\nicefrac 1 N)$ error dependence shown in~\cite{lacker2023sharp}, to improve existing theoretical guarantees? \textbf{(2)} Can we leverage recent advances in the theory of log-concave sampling to design better algorithms?

Our main proposal in this work is to \textbf{decouple} the error into two terms, representing the propagation of chaos and discretization errors respectively. This simple and \emph{modular} approach immediately allows us to answer both questions in the affirmative. 
Namely, we show how to combine established propagation of chaos bounds in various settings~\cite[including the sharp rate of][]{lacker2023sharp}
with a large class of sophisticated off-the-shelf log-concave samplers,
such as interacting versions of the randomized midpoint discretization of the underdamped Langevin dynamics \cite{shen2019randomized,he2020ergodicity}, Metropolis-adjusted algorithms \cite{chewi2021optimal,wu2022minimax, altschuler2023faster}, and the proximal sampler~\cite{LeeSheTia21RGO, chen22improved,fan2023improved}. 
Our framework yields improvements upon prior state-of-the-art, such as~\cite{bou2023nonlinear, fu2023mean}, and provides a clear path for future ones.

\subsection{Contributions and Organization}
\paragraph*{Propagation of chaos at stationarity.} 
We provide three propagation of chaos results which hold in the $\mc W_2$, $\sqrt{\KL}$, and $\sqrt{\FI}$ ``metrics''; the rates reflect the distance of the $k$-particle marginal of the finite-particle system from $\pi^{\otimes k}$:
(1) In the setting of~\eqref{eq:total_energy_with_interaction}, under strong displacement convexity, we obtain a $\mc O(\sqrt{\nicefrac{k}{N}})$ rate by adapting techniques from~\cite{sznitman1991topics, malrieu2001logarithmic}; (2) without assuming displacement convexity, but assuming a weaker interaction, we obtain the sharp rate of $\widetilde{\mc{O}}(\nicefrac{k}{N})$ following~\cite{lacker2023sharp}; (3) finally, in the general setting of~\eqref{eq:total_energy}, and assuming $\mc F$ is convex along linear interpolations, we obtain a $\mc O(\sqrt{\nicefrac{k}{N}})$ rate following~\cite{chen2022uniform}.

Unlike prior works, our proofs are carried out at stationarity; thus, our proofs are \emph{self-contained}, streamlined, and include various improvements (e.g., weaker assumptions and explicit bounds). 
As a result, our work also serves as a helpful exposition to the mathematics of propagation of chaos.

\paragraph*{Discretization.}
Once the error due to particle approximation is controlled, we then obtain improved complexity guarantees by applying recent advances in the theory of log-concave sampling to the finite-particle stationary distribution.
See Table~\ref{tab:rates_classical} for a summary of our results, and the discussion in \S\ref{scn:main_results} for comparisons with prior works and an application to neural network training. 

Once again, the importance of our framework is its \emph{modularity}, which allows for any combination of uniform-in-time propagation of chaos bounds and log-concave sampler, provided that the finite-particle stationary distribution satisfies certain isoperimetric properties needed for the sampling guarantees. Toward this end, we also provide tools for verifying these isoperimetric properties with constants that hold independently of the number of particles (see \S\ref{sec:mkv_isoperimetry}).

\subsection{Related Work}\label{scn:related}

\paragraph*{Mean-field equations.}
The McKean{--}Vlasov SDE was first formulated in the works \cite{mckean1966class, funaki1984certain, Mel1996Interacting}, with origins dating to much earlier~\cite{boltzmann1872more}. It has applications in many domains, from fluid dynamics \cite{villani2002review} to game theory \cite{lasry2007mean, carmona2018probabilistic}; see \cite{chaintron2022propagationpartone, chaintron2022propagationparttwo} for a comprehensive survey. The kinetic version of this equation is known as the Boltzmann equation, and propagation of chaos has similarly been studied under a variety of assumptions~\cite{bolley2010trend, monmarche2017long, guillin2021uniform, guillin2022convergence}. One prominent application within machine learning is the study of infinitely wide two-layer neural networks in the mean-field regime (see \S\ref{scn:nn_applications}).

\paragraph*{Propagation of chaos and sampling for \eqref{eq:total_energy_with_interaction}.}
The original propagation of chaos arguments of \cite{sznitman1991topics} were first made uniform in time in~\cite{malrieu2001logarithmic, malrieu2003convergence} in both entropy and $\mc W_2$. The aforementioned works all achieve an error of order $\widetilde{\mc{O}}(\sqrt{\nicefrac{k}{N}})$, and require a strong convexity assumption on $V$ and $W$. These were later adapted for non-smooth potentials~\cite{jabin2017mean, jabin2018quantitative, bresch2023mean}.
Finally, \cite{chen2022uniform} obtained an entropic propagation of chaos bound under a higher-order smoothness assumption. See \cite{chaintron2022propagationpartone} for a more complete bibliography.

The breakthrough result of~\cite{lacker2023hierarchies} obtained the sharp bound of $\widetilde{\mc O}(\nicefrac{k}{N})$ when the interaction is sufficiently weak, and this bound was made uniform in time in~\cite{lacker2023sharp}. Their approach differs significantly from previous proofs by considering a local analysis based on the recursive BBGKY hierarchy.
These results have been extended to other divergences, e.g., the $\chi^2$ divergence, but without a uniform-in-time guarantee~\cite{hess2023higher}. In addition, \cite{monmarche2024time} showed an extension of this result under a ``convexity at infinity" assumption.

The question of sampling from minimizers of \eqref{eq:total_energy_with_interaction} was first studied in \cite{talay1996probabilistic, bossy1997stochastic, antonelli2002rate}. These works focused on the Euler--Maruyama discretization of the finite-particle system~\eqref{eq:finite_particle_interactive}, under $L^\infty$-boundedness of the gradients. Subsequently, the convergence of the Euler--Maruyama scheme has been studied in many works, including but not limited to~\cite{bao2022approximations, dos2022simulation, li2023strong}. The strategy of disentangling finite particle error from time discretization also appears in~\cite{karimi2024stochastic}, which approaches the problem from the perspective of stochastic approximation. This work, however, is not focused on obtaining quantitative guarantees. Finally,~\cite{bou2023nonlinear} considered a non-linear version of Hamiltonian Monte Carlo; we give a detailed comparison with their work in \S\ref{scn:main_results}.

\paragraph*{Propagation of chaos and sampling for \eqref{eq:total_energy}.}
The mean-field (underdamped) Langevin algorithm for minimizing~\eqref{eq:total_energy} was proposed and studied in~\cite{chen2022uniform, Chen+24PoCKinetic}. Under alternative assumptions (see \S\ref{sec:mfl_chaos}), they established propagation of chaos with a $\mc{O}(\sqrt{\nicefrac k N})$ rate, for both the overdamped and the underdamped finite-particle approximations. Recent works from the machine learning community~\cite{nitanda2022convex, suzuki2022uniform, fu2023mean, suzuki2023convergence} studied the application of these algorithms for optimizing two-layer neural networks and obtained sampling guarantees.
We provide a detailed comparison with their works in \S\ref{scn:nn_applications}.

\section{Background and Notation}
Let $\Pac$ be the set of probability measures on $\R^{d}$ that admit a density with respect to the Lebesgue measure and have finite second moment. We will also abuse notation and use the same symbol for a measure and its density when there is no confusion. We use superscripts for the particle index, and subscripts for the time variable. We will use $\mc O, \widetilde{\mc O}$ to signify upper bounds up to numeric constants and polylogarithms respectively. 
We recall the definitions of convexity and smoothness:

\begin{definition}\label{defn:strcvx_smooth}
    A function $U: \R^d \to \R$ is $\alpha$-uniformly convex (allowing for $\alpha \le 0$) and $\beta$-smooth if the following hold respectively
    \begin{align*}
        \langle \nabla U(x) - \nabla U(y), x-y\rangle &\ge \alpha\,\norm{x-y}^2 &&\text{for all}~x,y\in\R^d\,, \\
        \norm{\nabla U(x) - \nabla U(y)} &\leq \beta\,\norm{x-y} &&\text{for all}~x,y\in\R^d\,.
    \end{align*}
\end{definition}

For two probability measures $\mu, \nu \in \Pac$, we define the $\msf{KL}$ divergence and the (relative) Fisher information by
\[
    \msf{KL}(\mu \mmid \nu) \defeq \E_\mu \bigl[\log \frac{\mu}{\nu} \bigr]
    \qquad \text{and} \qquad
    \msf{FI}(\mu \mmid \nu) \defeq \E_\mu\bigl[\bigl\lVert\nabla \log \frac{\mu}{\nu}\bigr\rVert^2\bigr]\,,
\]
with the convention $\msf{KL}(\mu \mmid \nu)= \msf{FI}(\mu \mmid \nu) = \infty$ whenever $\mu \not\ll\nu$.

We recall the definition of the log-Sobolev inequality, which is used both for propagation of chaos arguments as well as mixing time bounds.

\begin{definition}[{Log-Sobolev Inequality}]\label{def:lsi}
    A measure $\pi$ satisfies a log-Sobolev inequality with parameter $C_{\msf{LSI}}$ 
    if for all $\mu \in \Pac$,
    \begin{align}\label{eq:lsi}\tag{LSI}
    \msf{KL}(\mu \mmid \pi) \leq \frac{C_{\msf{LSI}}}{2}\, \msf{FI}(\mu \mmid \pi)\,.
    \end{align}
\end{definition}
When $\log(1/\pi)$ is $\alpha$-uniformly convex for $\alpha>0$, it follows from the Bakry{--}\'Emery condition that $\pi$ satisfies~\eqref{eq:lsi} with constant $C_{\msf{LSI}} \le 1/\alpha$~\cite[Proposition 5.7.1]{bakry2014analysis}.

We can also define the $p$-Wasserstein distance $\mc W_p(\mu, \pi)$, $p \ge 1$, between $\mu, \pi$ as
\begin{align*}
    \mc{W}_p^p(\mu, \pi) = \inf_{\gamma \in \Gamma(\mu, \pi)} \int \norm{x - y}^p \, \gamma(\D x, \D y),
\end{align*}
where $\Gamma(\mu, \pi)$ is the set of all joint probability measures on $\R^d \hspace{-3pt}\times\hspace{-2pt} \R^d$ with marginals $\mu, \pi$ respectively.

Lastly, we recall that the celebrated Otto calculus interprets the space $\Pac$, equipped with the $\mc W_2$ metric, as a formal Riemannian manifold \cite{otto2001geometry}. In particular, the Wasserstein gradient of a functional $\mc L : \Pac\to\R \cup\{\infty\}$ is given as $\nabla_{\mc W_2} \mc L = \nabla \delta \mc L$.
Here, $\delta \mc L$ is the first variation defined as follows: for all $\nu_0,\nu_1 \in \Pac$, $\delta \mc L(\nu_0) : \R^d\to\R$ satisfies
\begin{align*}
   \lim_{t \searrow 0} \frac{\mc L((1-t)\,\nu_0 + t\,\nu_1) - \mc L(\nu_0)}{t} = \langle \delta \mc L(\nu_0), \nu_1 - \nu_0 \rangle
   \deq \int \delta \mc L(\nu_0)\,\D(\nu_1 - \nu_0)\,.
\end{align*}
The first variation is defined up to an additive constant, but the Wasserstein gradient is unambiguous.
See~\cite{AGS} for a rigorous development.
As a shorthand, we will write $\delta \mc L(\nu_0, x) \deq \delta \mc L(\nu_0)(x)$ and similarly $\nabla_{\mc W_2} \mc L(\nu_0, x) \deq \nabla_{\mc W_2} \mc L(\nu_0)(x)$.

\subsection{SDE Systems and Their Stationary Distributions}\label{scn:FPE&MFE}

\subsubsection{The Pairwise McKean{--}Vlasov Setting}\label{sec:mckean_vlasov}

In the formalism introduced in the previous section, we note that~\eqref{eq:mean_field_interactive} can be interpreted as Wasserstein gradient flow for~\eqref{eq:total_energy_with_interaction}.
In this paper, we refer to~\eqref{eq:mean_field_interactive} as the \emph{pairwise McKean{--}Vlasov} process. 
As noted in the introduction, it has the stationary distribution~\eqref{eq:mean_field_stat_interactive} which minimizes~\eqref{eq:total_energy_with_interaction}.

Recall also that the equation~\eqref{eq:mean_field_interactive} is the mean-field limit of the finite-particle system~\eqref{eq:finite_particle_interactive}. This $N$-particle system has the following stationary distribution: for $x^{1:N} = [x^1,\dotsc, x^N]\in \R^{d\times N}$,
\begin{align}\label{eq:Nparticle_stat_interactive}
    \mu^{1:N}(x^{1:N}) \propto \exp\Bigl( - \frac{2}{\sigma^2} 
    \sum_{i \in [N]}
    V(x^i) 
    - \frac{1}{\sigma^2\,(N-1)} 
    \sum_{i \in [N]} \sum_{j \in [N]\setminus i} 
    W(x^i - x^j)\Bigr)\,.
\end{align}
The system~\eqref{eq:finite_particle_interactive} can be viewed as an approximation to \eqref{eq:mean_field_interactive}, with the expectation term in the drift replaced by an empirical average. Note that the measure $\mu^{1:N}$ is exchangeable.\footnote{Exchangeability refers to the property that the law of $[x^{1}, \ldots, x^{N}]$ equals the law of $[x^{\sigma(1)}, \ldots, x^{\sigma(N)}]$ for any permutation $\sigma$ of $\{1, \ldots, N\}$.} While the standard approach is to apply an Euler--Maruyama discretization to \eqref{eq:finite_particle_interactive} in order to sample from \eqref{eq:mean_field_interactive}, our perspective is to write more sophisticated samplers for $\mu^{1:N}$ directly.
Indeed, unlike~\eqref{eq:mean_field_stat_interactive}, the finite-particle stationary distribution~\eqref{eq:Nparticle_stat_interactive} is explicit and amenable to sampling methods.

\subsubsection{The General McKean{--}Vlasov Setting}\label{sec:mfl}

More generally, we consider the functional~\eqref{eq:total_energy} where $\mc F$ is of the form $\mc F(\mu) = \mc F_0(\mu) + \frac{\lambda}{2} \int \norm{\cdot}^2 \, \D \mu$ with $\lambda \geq 0$.
The second term acts as regularization and is common in the literature \cite{fu2023mean, suzuki2023convergence}.
We can describe its Wasserstein gradient flow as the marginal law of a particle trajectory satisfying the following SDE\@, which we call the \emph{general McKean{--}Vlasov equation}: 
\begin{align}\label{eq:mean_field} \tag{$\msf{gMV}$}
    \D X_t = \{-\nabla_{\mc W_2}\mc F_0(\pi_t, X_t) - \lambda X_t\}\,\D t + \sigma \, \D B_t\,,
\end{align}
where $\pi_t = \law(X_t)$, and $\{B_t\}_{t\ge 0}$ is a standard Brownian motion on $\R^d$. 
The stationary distribution $\pi$ of \eqref{eq:mean_field}, and its linearization $\pi_\mu$ around a measure $\mu \in \Pac$, satisfy the following equations: 
\begin{align}\label{eq:mean_field_stat}
  \!\!  \pi(x) \propto \exp \Bigl(-\frac{2}{\sigma^2}\,\delta \mc F_0(\pi,x) - \frac{\lambda\,\norm x^2}{\sigma^2} \Bigr) \:\; \text{and} \:\;
    \pi_{\mu}(x) \propto \exp\Bigl(-\frac{2}{\sigma^2}\, \delta \mc F_0(\mu, x)- \frac{\lambda\,\norm x^2}{\sigma^2}\Bigr)\,.
\end{align}
The latter is called the \textit{proximal Gibbs distribution} with respect to $\mu$. The general dynamics corresponds to the mean-field limit of the following finite-particle system described by an $N$-tuple of stochastic processes $\{X_t^{1:N}\}_{t\ge 0}  \deq  \{(X^1_t,\dotsc, X^N_t)\}_{t \geq 0}$:
\begin{align}\label{eq:finite_particle} \tag{$\msf{gMV}_N$}
    \D X_t^i = \{-\nabla_{\mc W_2} \mc F_0(\rho_{X_t^{1:N}}, X_t^i) - \lambda X_t^i\}\,\D t + \sigma \, \D B_t^i\,,
\end{align}
and $\rho_{x^{1:N}} = \frac{1}{N} \sum_{i=1}^N \delta_{x^i}$ 
is the empirical measure of the particle system. 
The stationary distribution for~\eqref{eq:finite_particle} is given as follows~\cite[(2.16)]{chen2022uniform}: for $x^{1:N} = [x^1,\dotsc, x^N]\in \R^{d\times N}$,
\begin{align}\label{eq:Nparticle_stat}
    \mu^{1:N}(x^{1:N}) \propto \exp\Bigl( - \frac{2N}{\sigma^2}\, \mc F_0(\rho_{x^{1:N}}) - \frac{\lambda}{\sigma^2}\, \norm{x^{1:N}}^2\Bigr)\,.
\end{align}
One can show that $\nabla_{x^i} \mc F_0(\rho_{x^{1:N}}) = \frac{1}{N} \,\nabla_{\mc W_2} \mc F_0(\rho_{x^{1:N}}, x^i)$, and hence~\eqref{eq:finite_particle} is simply the Langevin diffusion corresponding to stationary measure~\eqref{eq:Nparticle_stat}.
Moreover, when $\lambda = 0$ and choosing $\mc F_0(\mu) = \int V(x)\,\mu(\D x) + \iint W(x-y)\,\mu(\D x)\,\mu(\D y)$, then the equations~\eqref{eq:mean_field},~\eqref{eq:mean_field_stat},~\eqref{eq:finite_particle}, and~\eqref{eq:Nparticle_stat} reduce to~\eqref{eq:mean_field_interactive},~\eqref{eq:mean_field_stat_interactive},~\eqref{eq:finite_particle_interactive}, and~\eqref{eq:Nparticle_stat_interactive}, respectively. 

\section{Technical Ingredients}
Our general approach for sampling from the stationary distribution $\pi$ in either~\eqref{eq:mean_field_stat_interactive} or~\eqref{eq:mean_field_stat} is to directly apply an off-the-shelf sampler for the finite-particle stationary distribution $\mu^{1:N}$. The theoretical guarantees for this procedure require two main ingredients: (1) control of the ``bias''---i.e. the error incurred by approximating $\pi$ by the $1$-particle marginal of $\mu^{1:N}$---and (2) verification of isoperimetric properties which allow for fast sampling from the measure $\mu^{1:N}$.

\subsection{Bias Control via Uniform-in-Time Propagation of Chaos}\label{scn:bias_control}

In this section, we focus on the first ingredient, namely, obtaining control of the bias via uniform-in-time propagation of chaos results.
Proofs for this section are given in \S\ref{app:bias_control}.

\subsubsection{Pairwise McKean{--}Vlasov Setting}\label{sec:mkv_chaos}

We first consider the pairwise McKean{--}Vlasov setting described in \S\ref{sec:mckean_vlasov}.
Our first propagation of chaos result uses the following three assumptions.
\begin{assumption}\label{as:smoothness}
The potentials $V, W$ are $\beta_V, \beta_W$-smooth respectively.
\end{assumption}

\begin{assumption}\label{as:pi_lsi}
The distribution $\pi$ satisfies~\eqref{eq:lsi} with parameter $C_{\msf{LSI}}(\pi)$.
\end{assumption}

\begin{assumption}\label{as:weak_interaction}
The ratio $\rho\defeq \nicefrac{\sigma^4}{8\beta_W^2 C_{\msf{LSI}}^2(\pi)}$ is at least $3$.
\end{assumption}

\begin{remark}
    Note that from \eqref{eq:mean_field_stat_interactive}, we typically would expect $C_{\msf{LSI}}^2(\pi)$ to also scale as $\sigma^4$ (e.g., in the case when $V$ and $W$ are $\alpha$-uniformly convex for $\alpha>0$). Therefore, Assumption~\ref{as:weak_interaction} is typically invariant to the scaling of $\sigma$ and can be satisfied even for $\sigma \searrow 0$.
\end{remark}
Under these assumptions, we obtain a sharp propagation of chaos result via a similar argument as \cite{lacker2022quantitative, lacker2023sharp}. We note that the former is more permissive regarding the constant in Assumption~\ref{as:weak_interaction} as compared to this work.

\begin{theorem}[{Sharp Propagation of Chaos}]\label{thm:N_choice_bias}
     Under Assumptions~\ref{as:smoothness}, \ref{as:pi_lsi} and \ref{as:weak_interaction}, for any $N \ge 100$ and $k \in [N]$, it holds that $\msf{KL}(\mu^{1:k} \mmid \pi^{\otimes k}) = \widetilde{\mc O}(dk^2/N^2)$.
     Thus, $\msf{KL}(\mu^{1:k} \mmid \pi^{\otimes k}) < \varepsilon^2$ if
        \begin{align}\label{ex:N_choice_bias}
            N \ge 100 \vee \widetilde{\Omega}\bigl(k\sqrt{d}\,\epsilon^{-1}\bigr)\,.
        \end{align}
\end{theorem}

We note that the rate in Theorem~\ref{thm:N_choice_bias} is sharp; see the Gaussian case in Example~\ref{eg:Gaussian}. 
A condition such as Assumption~\ref{as:weak_interaction} is in general necessary, since otherwise the minimizer of~\eqref{eq:total_energy_with_interaction} may not even be unique~\cite[see the example and discussion in][]{lacker2023sharp}.
However, it can be restrictive, as it requires the interaction to be sufficiently weak.
With the following convexity assumption, we can obtain a propagation of chaos result without Assumption~\ref{as:weak_interaction}.

\begin{assumption}\label{as:str_cvx_VW}
    The potentials $V, W$ are $\alpha_V,\alpha_W$-uniformly convex with $\alpha_V + {\alpha_W^-} > 0$.
    Here, ${\alpha_W^-} \deq \alpha_W \wedge 0$ denotes the negative part of $\alpha_W$. 
\end{assumption}

The following weaker result consists of two parts. The first, a Wasserstein propagation of chaos result, is based on~\cite{sznitman1991topics}. The second, building on the first, is a uniform-in-time entropic propagation of chaos bound following from a Fisher information bound.
The arguments are similar to those in~\cite{malrieu2001logarithmic, malrieu2003convergence}, albeit simplified (since we work at stationarity) and presented here with explicit constants. 

\begin{theorem}[{Weak Propagation of Chaos}]\label{thm:weak_prop_chaos}
    Under Assumptions~\ref{as:smoothness} and~\ref{as:str_cvx_VW}, for any $N \ge \frac{\alpha_V - {\alpha_W^-}}{\alpha_V + {\alpha_W^-}} \vee 2,$
    if we denote $\alpha \deq \alpha_V + {\alpha_W^-}$, then 
    \begin{align}
        \mc W_2^2(\mu^{1:k}, \pi^{\otimes k}) &\leq \frac{4\beta_W^2 \sigma^2 d}{\alpha ^3}\,\frac k N\,, \label{eq:w2_weak_prop} \\
        \msf{KL}(\mu^{1:k} \mmid \pi^{\otimes k}) \leq \frac{\sigma^2}{4\alpha} \FI(\mu^{1:k} \mmid \pi^{\otimes k})
        &\le \frac{132\beta_W^2\,{(\beta_V+\beta_W)}^2\, d}{\alpha^4 }\,\frac{k}{N}\,. \label{eq:kl_prop}
    \end{align}
\end{theorem}

\subsubsection{General McKean{--}Vlasov Setting}\label{sec:mfl_chaos}

In the more general case where we aim to minimize~\eqref{eq:total_energy} for a generic functional $\mc F$ of the form $\mc F(\mu) = \mc F_0(\mu) + \frac{\lambda}{2} \int \norm{\cdot}^2 \, \D \mu$, we impose the following assumptions. They can be largely seen as generalizations of the conditions for the pairwise case, and they are inherited from~\cite{chen2022uniform, suzuki2023convergence}. There is an additional convexity condition (Assumption~\ref{as:functional_convexity}), which in the pairwise McKean{--}Vlasov setting amounts to positive semidefiniteness of the kernel $(x,y) \mapsto W(x-y)$ on $\R^d \times \R^d$; thus, in general, the following assumptions are incomparable with the ones in \S\ref{sec:mkv_chaos}.

\begin{assumption}\label{as:functional_convexity}
    The functional $\mc F_0$ is convex in the usual sense. For all $\nu_0, \nu_1 \in \Pac$, $t \in [0,1]$,
    \begin{align*}
        \mc F_0((1-t)\, \nu_0 + t\, \nu_1) \leq (1-t)\, \mc F_0(\nu_0) + t\, \mc F_0(\nu_1)\,.
    \end{align*}
\end{assumption}

\begin{assumption}\label{as:functional_smoothness}
    The functional $\mc F_0$ is smooth 
    in the sense that for all $x, y \in \R^d$, $\nu, \nu' \in \Pac$, there is a uniform constant $\beta$ such that
    \begin{align*}
        \norm{\nabla_{\mc W_2} \mc F_0(\nu, x) - \nabla_{\mc W_2} \mc F_0(\nu', y)} \leq \beta\, (\norm{x-y} + \mc W_1(\nu, \nu'))\,.
    \end{align*}
\end{assumption}

\begin{assumption}\label{as:prox_gibbs_lsi}
    The proximal Gibbs measures satisfy~\eqref{eq:lsi} with a uniform constant: namely, it holds that $C_{\msf{LSI}}(\pi) \vee \sup_{\mu\in\mc P_2(\R^d)} C_{\msf{LSI}}(\pi_\mu)\le \overline C_{\msf{LSI}}$.
\end{assumption}

\begin{remark}
    These assumptions taken together cover settings not covered in the preceding sections, including optimization of two-layer neural networks. See~\cite[Remark 3.1]{chen2022uniform} and \S\ref{scn:nn_applications}.
\end{remark}

Under these assumptions, we can derive an entropic propagation of chaos bound by following the proof of~\cite{chen2022uniform}. Through a tighter analysis, we are able to reduce the dependence on the condition number $\kappa \deq \overline{C}_{\msf{LSI}} \beta/\sigma^2$ from $\kappa^2$ to $\kappa$. 

\begin{theorem}[Propagation of Chaos for General Functionals]
\label{thm:entropic_prop_chaos_general_functional}
    Under Assumptions~\ref{as:functional_convexity},~\ref{as:functional_smoothness}, and~\ref{as:prox_gibbs_lsi}, for $N \geq \nicefrac{160 \beta \overline C_{\msf{LSI}}}{\sigma^2},$
    we have
    \[
        \frac{1}{2\overline C_{\msf{LSI}}} \mc W_2^2(\mu^{1:k},\pi^{\otimes k}) \le \KL(\mu^{1:k} \mmid \pi^{\otimes k}) \leq \frac{33 \beta \overline C_{\msf{LSI}} dk}{\sigma^2 N}\,.
    \]
\end{theorem}

Among these assumptions, the hardest to verify is the uniform LSI of Assumption~\ref{as:prox_gibbs_lsi}.
Following~\cite{suzuki2023convergence}, we introduce the following sufficient condition for the validity of Assumption~\ref{as:prox_gibbs_lsi}; see Lemma~\ref{lem:uniform_lsi} for a more precise statement.

\begin{assumption}\label{as:functional_bdd_grad}
There exists a uniform bound on the Wasserstein gradient of the interaction term $\mc F_0$: for some constant $B < \infty$ and all $\mu \in \Pac$, $x \in \R^d$,
    $\norm{\nabla_{\mc W_2} \mc F_0(\mu, x)} \leq B\,.$
\end{assumption}

\begin{lemma}[Informal]\label{lem:prox_gibbs_lsi_informal}
    Assumptions~\ref{as:functional_smoothness} and~\ref{as:functional_bdd_grad} 
    imply Assumption~\ref{as:prox_gibbs_lsi} with an explicit constant $\overline C_{\msf{LSI}}$, given in terms of $B$, $\beta$, $\lambda$, and $\sigma$.
\end{lemma}

\subsection{Isoperimetric Properties of the Stationary Distributions}\label{scn:isoperimetry}

In this section, we verify the isoperimetric properties of $\pi, \mu^{1:N}$ in the~\eqref{eq:mean_field_interactive} setting, with proofs provided in \S\ref{app:iso_inv_measure}.

\subsubsection{Pairwise McKean{--}Vlasov Setting}\label{sec:mkv_isoperimetry}

If $V$, $W$ satisfy Assumptions~\ref{as:smoothness} and~\ref{as:str_cvx_VW} (i.e. $V$ and $W$ have bounded Hessians), then the potential for~\eqref{eq:mean_field_stat_interactive}, i.e. $\log(1/\pi)$, 
is  $\frac{2}{\sigma^2}\,(\alpha_V + \alpha_W)$-convex and $\frac{2}{\sigma^2}\,(\beta_V + \beta_W)$-smooth.
By the Bakry{--}\'{E}mery condition, $\pi$ satisfies~\eqref{eq:lsi} with parameter $C_{\msf{LSI}}(\pi) \le \nicefrac{\sigma^2}{2\,(\alpha_V + \alpha_W)}$.

Similarly, for the invariant measure $\mu^{1:N}$ in~\eqref{eq:Nparticle_stat_interactive}, we can prove the following.

\begin{lemma}\label{lem:strongCvxCase}
    If $V$ and $W$ satisfy Assumption~\ref{as:smoothness}, then $\log(1/\mu^{1:N})$ is $\frac{2}{\sigma^2}\,(\beta_V + \frac{N}{N-1}\,\beta_W)$-smooth.
    
    If $V$ and $W$ satisfy Assumption~\ref{as:str_cvx_VW}, then $\log(1/\mu^{1:N})$ is $\frac{2}{\sigma^2}\,(\alpha_V + \frac{N}{N-1}\,{\alpha_W^-})$-convex.\footnote{Only the negative part of $\alpha_W$ contributes to the strong log-concavity of $\mu^{1:N}$.
    This is consistent with~\cite[Theorem 5.15]{villani2003topics}, which asserts that when $\alpha_W > 0$, the interaction energy $\mu \mapsto \iint W(x-y)\,\mu(\D x) \,\mu(\D y)$ is $\alpha_W$-strongly displacement convex over the subspace of probability measures with fixed mean, but only weakly convex
    over the full Wasserstein space.}
\end{lemma}

We now consider the non-log-concave case. It is standard in the sampling literature that the assumption of~\eqref{eq:lsi} for the stationary distribution yields mixing time guarantees.
Since our strategy is to sample from~\eqref{eq:Nparticle_stat_interactive}, we therefore seek an LSI for $\mu^{1:N}$, formalized as the following assumption.
\begin{assumption}\label{as:lsi_N}
    The distribution $\mu^{1:N}$ satisfies~\eqref{eq:lsi} with parameter $C_{\msf{LSI}}(\mu^{1:N})$.
\end{assumption}

In this section, we provide an easily verifiable condition, combining a Holley{--}Stroock condition~\cite{HolStr1987LSI} with a weak interaction condition, for this assumption to hold with an $N$-independent constant.

\begin{assumption}\label{as:bdd_perturb_VW}
    The potentials $V$ and $W$ can be decomposed as $V = V_0 + V_1$ and $W = W_0 + W_1$ such $V_0$, $W_0$ satisfy Assumption~\ref{as:str_cvx_VW} and $\osc(V_1), \osc(W_1) < \infty$, where for a function $U : \R^d\to\R$ we define $\osc(U) \deq \sup U - \inf U$. Furthermore, the following weak interaction condition holds:
    \begin{align*}
        \frac{\sigma^2}{\beta_W \overline C_{\msf{LSI}}} \ge \sqrt 6\,, \qquad \text{where}\qquad \overline C_{\msf{LSI}} \deq \frac{\sigma^2}{\alpha_{V_0} + \frac N {N-1}\,{\alpha_{W_0}^-}} \exp\Bigl(\frac{2}{\sigma^2}\, \bigl(\osc(V_1) +\osc(W_1)\bigr)\Bigr)\,.
    \end{align*}
\end{assumption}
A careful application of the Holley{--}Stroock perturbation principle yields the following lemma.

\begin{lemma}\label{lem:bdd_perturb_meanfield}
    Under Assumption~\ref{as:bdd_perturb_VW}, $\pi$, $\mu^{1:N}$ satisfy~\eqref{eq:lsi} with parameters
    \begin{align}
        C_{\msf{LSI}}(\pi) &\leq \frac{\sigma^2}{2\,(\alpha_{V_0} + \alpha_{W_0})} \exp\Bigl(\frac{2}{\sigma^2}\,\bigl(\osc(V_1) + \osc(W_1)\bigr)\Bigr)
        \le \frac{1}{2}\,\overline C_{\msf{LSI}}\,,\label{eq:clsi_pi_est} \\
        C_{\msf{LSI}}(\mu^{1:N}) &\le \overline C_{\msf{LSI}}\nonumber\,.
    \end{align}
    In particular, Assumption~\ref{as:weak_interaction} holds.
\end{lemma}

\subsubsection{General McKean{--}Vlasov Setting}

In the setting~\eqref{eq:total_energy} with $\mc F(\mu) = \mc F_0(\mu) + \frac{\lambda}{2} \int \norm{\cdot}^2 \, \D \mu$, we verify that Assumption~\ref{as:functional_bdd_grad}
yields~\eqref{eq:lsi} for $\pi$. See Corollary~\ref{cor:unif_LSI_crazy} for a more precise statement.

\begin{lemma}[Informal]\label{lem:lsiN_informal}
    In the mean-field Langevin setting of \S\ref{sec:mfl}, suppose that Assumption~\ref{as:functional_bdd_grad} holds.
    Then, Assumption~\ref{as:lsi_N} holds with $C_{\msf{LSI}}(\mu^{1:N})$ depending on $d$, $B$, $\beta$, $\lambda$, and $\sigma$, but not on $N$.
\end{lemma}

Obtaining Lemma~\ref{lem:lsiN_informal} is not straightforward, and we rely on a novel argument combining heat flow estimates from~\cite{BriPed24HeatFlow} with a generalized version of the propagation of chaos result in Theorem~\ref{thm:entropic_prop_chaos_general_functional}; see \S\ref{scn:functional_lsi} for details.

\section{Sampling from the Mean-Field Target}\label{scn:main_results}

In this section, we present results for sampling from $\pi$. As outlined in Algorithm~\ref{alg:particle-sampler}, we use off-the-shelf log-concave samplers to sample from $\mu^{1:N}$, during which we access the first-order\footnote{For our results involving the proximal sampler, we also assume access to a proximal oracle for simplicity.} oracle for $\mu^{1:N}$ (i.e. an oracle for evaluation of $\log \mu^{1:N}$ up to an additive constant, and for evaluation of $\nabla \log \mu^{1:N}$).
For $N$ sufficiently large, the first particle given by Algorithm~\ref{alg:particle-sampler} is approximately distributed according to $\pi$: for $\hat \mu^{1:N}$ the law of the output of the log-concave sampler and its $1$-particle marginal distribution $\hat \mu^{1}$,
\[
    \mc W_2(\hat \mu^{1}, \pi) 
    \le \mc W_2(\hat \mu^{1}, \mu^{1}) + \mc W_2(\mu^{1}, \pi)
    \le \sqrt{\frac 1 N}\, \mc W_2(\hat \mu^{1:N}, \mu^{1:N}) + \mc W_2(\mu^{1}, \pi)\,,
\]
where the inequality follows from exchangeability (Lemma~\ref{lem:wass_exchangeable}).
A similar decomposition also holds for $\KL$, although the argument is more technical. We defer its presentation to \S\ref{app:sampling}.

\begin{algorithm}[h]
\hspace*{\algorithmicindent} \textbf{Input:} the number $N$ of total particles, a log-concave sampler $\msf{LC{\text{-}}Sampler}$\\
\hspace*{\algorithmicindent} \textbf{Output:} $k$ particles $\hat X^{1:k}$
\begin{algorithmic}[1]
\caption{Sampling from the Mean-Field Stationary Distribution}\label{alg:particle-sampler}
\STATE Sample $\hat X^{1:N} \sim \hat \mu^{1:N}$ via $\msf{LC{\text{-}}Sampler}$, so that $\hat \mu^{1:N} \approx \mu^{1:N}$, e.g., in $\mc W_2$ or $\sqrt{\KL}$.
\STATE Output the first $k$ particles $\hat X^{1:k}$.
\end{algorithmic}
\end{algorithm}

To bound the second term by $\epsilon$, it suffices to choose $N$ according to the propagation of chaos results in \S\ref{scn:bias_control}.
Our results are summarized in Table~\ref{tab:rates_classical}, in which we record \textbf{the total number of oracle calls $M$ for $\mu^{1:N}$ made by the sampler $M$} and \textbf{the number of particles $N$} needed to achieve $\varepsilon$ error in the desired metric, hiding polylogarithmic factors.
Note that in the pairwise McKean{--}Vlasov setting, each oracle call to $\mu^{1:N}$ requires $N$ calls to an oracle for $V$, and $\binom{N}{2}$ calls to an oracle for $W$.

\begin{table}
    \centering
    \begin{tabular}{c|c|c|c|c}
         \hline {\textbf{Algorithm}} & {\textbf{``Metric"}} & {\textbf{Assumptions}} & $M$ & $N$  \\
         \hline LMC & \multirow{3}{*}{$\nicefrac{\sqrt\alpha}{\sigma}\,\mc W_2$} & \multirow{3}{*}{\ref{as:smoothness},~\ref{as:pi_lsi},~\ref{as:weak_interaction},~\ref{as:lsi_N}} & ${\kappa^2 d}/{\varepsilon^2}$ & \multirow{3}{*}{${d^{1/2}}/{\varepsilon}$} \\
          MALA--PS &&& $\kappa d^{3/4}/\varepsilon^{1/2}$ & \\
          ULMC--PS &&& ${\kappa^{3/2} d^{1/2}}/{\epsilon}$ & \\
          \hline ULMC${}^+$ & $\nicefrac{\sqrt\alpha}{\sigma}\,\mc W_2$ & \ref{as:smoothness},~\ref{as:weak_interaction},~\ref{as:str_cvx_VW} & ${\kappa d^{1/3}}/{\varepsilon^{2/3}}$ & ${d^{1/2}}/{\varepsilon}$ \\
          \hline LMC & \multirow{2}{*}{$\sqrt{\KL}$} &  \multirow{4}{*}{\ref{as:smoothness},~\ref{as:str_cvx_VW}} & ${\kappa^2 d}/{\varepsilon^2}$ &\multirow{2}{*}{${\kappa^4 d}/{\varepsilon^2}$} \\
           ULMC && & ${\kappa^{3/2} d^{1/2}}/{\varepsilon}$ & \\
           \cline{1-2}\cline{4-5} LMC & \multirow{2}{*}{$\nicefrac{\sqrt\alpha}{\sigma}\,\mc W_2$} & &${\kappa d}/{\varepsilon^2}$ & \multirow{2}{*}{${\kappa^2 d}/{\varepsilon^2}$} \\
           ULMC${}^+$ && & ${\kappa d^{1/3}}/{\varepsilon^{2/3}}$ &  \\
           \hline LMC & \multirow{2}{*}{$\sqrt{\KL}$} &  \multirow{2}{*}{\ref{as:functional_convexity},~\ref{as:functional_smoothness},~\ref{as:prox_gibbs_lsi},~\ref{as:lsi_N}} & ${\kappa^2 d}/{\varepsilon^2} $& \multirow{2}{*}{${\kappa d}/{\varepsilon^2}$} \\
           ULMC{--}PS && & $\kappa^{3/2} d^{1/2}/\varepsilon$ &
    \end{tabular}
    \caption{\small In this table, we record $M$, the total number of oracle queries to $\log \mu^{1:N}$ made by the log-concave sampler, and $N$, the number of particles.}\label{tab:rates_classical}
\end{table} 
The algorithms in the table refer to: Langevin Monte Carlo (LMC); underdamped Langevin Monte Carlo (ULMC); discretizations of the underdamped Langevin diffusion via the randomized midpoint method~\cite{shen2019randomized} or the shifted ODE method~\cite{FosLyoObe21ShiftedODE} (ULMC${}^+$); and implementation of the proximal sampler~\cite{LeeSheTia21RGO, chen22improved} via the Metropolis-adjusted Langevin algorithm or via ULMC (MALA{--}PS and ULMC{--}PS respectively).
Note that LMC applied to sample from $\mu^{1:N}$ is simply the Euler{--}Maruyama discretization of~\eqref{eq:finite_particle_interactive}, and likewise ULMC is the algorithm considered in~\cite{fu2023mean}.
We refer to \S\ref{app:sampling} for proofs and references.

To streamline the rates, we simplify the notation by defining $\beta = \beta_V + \beta_W$ if Assumption~\ref{as:smoothness} holds, otherwise we use the value from Assumption~\ref{as:functional_smoothness}. 
We let $\alpha = \alpha_V + {\alpha_W^-}$ under Assumption~\ref{as:str_cvx_VW}, $\alpha = \nicefrac{\sigma^2}{2 \max\{C_{\msf{LSI}}(\mu^{1:N}), C_{\msf{LSI}}(\pi)\}}$ under Assumptions~\ref{as:pi_lsi} and~\ref{as:lsi_N}, and $\alpha = \nicefrac{\sigma^2}{2\max\{C_{\msf{LSI}}(\mu^{1:N}), \overline C_{\msf{LSI}}\}}$ in the general McKean{--}Vlasov setting.

Finally, we let $\kappa \deq \beta/\alpha$ denote the condition number. We briefly justify this terminology. If the target and all proximal Gibbs measures are strongly convex with parameter $\alpha/\sigma^2$, then the Bakry--\'Emery condition implies that $\overline{C}_{\msf{LSI}} \le \tfrac{\sigma^2}{\alpha}$.
Hence, the scale-invariant ratio $\overline C_{\msf{LSI}} \beta/\sigma^2$ reduces to the classical condition number $\beta/\alpha$, the ratio of the largest to smallest eigenvalues of the Hessian matrices for $V$ and $W$.
Therefore, $\overline C_{\msf{LSI}} \beta/\sigma^2$ is a generalization of the condition number to settings beyond uniform strong convexity which allows us to state more interpretable bounds. The additional assumption $\kappa \leq \sqrt{d}/\epsilon$ will be used to simplify some of the rates.

In the following subsections, we discuss some of the results in greater detail.

\subsection{Pairwise McKean{--}Vlasov Setting}

\begin{example}[Gaussian Case]\label{eg:Gaussian}
    Consider a quadratic confinement and interaction,
    \[
        V(x) = \frac{1}{2}\, x^\T A x=\frac{1}{2}\,\norm{x}_A^2\,,    \qquad W(x) = \frac{\lambda}{2}\,\norm x^2\,,
    \]
    for some matrix $A \in \R^{d \times d}$ with $A \succ 0$, $\lambda \geq 0$. 
    The resulting stationary distributions can be calculated explicitly to be Gaussians.
    We show in \S\ref{sec:cal-Gaussian} that for large $N$, 
    $
    \KL(\mu^{1:k}\mmid \pi^{\otimes k}) = \widetilde{\Theta} (\nicefrac{dk^2}{N^2}\,)$.
    This shows that the rate in Theorem~\ref{thm:N_choice_bias} is sharp.
\end{example}

\begin{example}[Strongly Convex Case]
    Consider the strongly convex case where $\alpha = \alpha_V + {\alpha_W^-} > 0$.
    The prior work~\cite{bou2023nonlinear} also considered the problem of sampling from the mean-field stationary distribution $\pi$, with $\sigma^2 = 2$. If we count the number of calls to a gradient oracle for $V$, their complexity bound reads $\widetilde{\mc O}(\nicefrac{\kappa^{5/3} d^{4/3}}{\varepsilon^{8/3}})$ to achieve $\nicefrac{\sqrt \alpha}{\sigma}\,\mc W_1(\hat\mu^1, \pi) \le \varepsilon$.
    We note that their assumptions are not strictly comparable to ours. They require the interaction $W$ to be sufficiently weak, in the sense that $\beta_W \lesssim \alpha$, which is similar\footnote{See eq.\ (2.24) therein; note that they have a scaling factor of $\epsilon$ in front of their interaction term, so that our parameter $\beta_W$ is equivalent to their $\varepsilon \tilde L$.} to our Assumption~\ref{as:weak_interaction}; on the other hand, they only assume $\alpha_V > 0$, rather than $\alpha_V + {\alpha_W^-} > 0$.
    Nevertheless, we attempt to make some comparisons with their work below.

    \textbf{Without} Assumption~\ref{as:weak_interaction}, ULMC${}^+$ achieves $\nicefrac{\sqrt\alpha}{\sigma}\,\mc W_2(\hat\mu^1,\pi) \le \varepsilon$ with complexity $\widetilde{\mc O}(\nicefrac{\kappa^3 d^{4/3}}{\varepsilon^{8/3}})$, which matches the guarantee of~\cite{bou2023nonlinear} up to the dependence on $\kappa$.
    We can also obtain guarantees in $\sqrt{\KL}$, at the cost of an extra factor of $\kappa^2$.

    \textbf{With} Assumption~\ref{as:weak_interaction}, MALA{--}PS has complexity $\widetilde{\mc O}(\nicefrac{\kappa d^{5/4}}{\varepsilon^{3/2}})$ and ULMC${}^+$ has complexity $\widetilde{\mc O}(\nicefrac{\kappa d^{5/6}}{\varepsilon^{5/3}})$, which improve substantially upon~\cite{bou2023nonlinear}.

    To summarize, in the strongly convex case, we have obtained numerous improvements: (i) we can obtain results even without the weak interaction condition (Assumption~\ref{as:weak_interaction}); (ii) when we assume the weak interaction condition, we obtain improved complexities; (iii) our results hold in stronger metrics; (iv) our approach is generic, allowing for the consideration of numerous different samplers without needing to establish new propagation of chaos results (by way of comparison,~\cite{bou2023nonlinear} developed a tailored propagation of chaos argument for their non-linear Hamiltonian Monte Carlo algorithm).
\end{example}

\begin{example}[Bounded Perturbations]
    Both the results of~\cite{bou2023nonlinear} as well as our own allow for non-convex potentials, albeit under different assumptions---\cite{bou2023nonlinear} require strong convexity at infinity, whereas we require~\eqref{eq:lsi} for the stationary measures $\mu^{1:N}$ and $\pi$.
    In order to obtain sampling guarantees with low complexity, it is important for the LSI constant of $\mu^{1:N}$ to be independent of $N$.
    We have provided a sufficient condition for this to hold: $V$ and $W$ are bounded perturbations of $V_0$ and $W_0$ respectively, where $\alpha_{V_0} + {\alpha_{W_0}^-} > 0$; see Lemma~\ref{lem:bdd_perturb_meanfield}.

    We also note that in this setting, both of our works require a weak interaction condition. This is in general necessary in order to ensure uniqueness of the mean-field stationary distribution, see the discussion in \S\ref{sec:mkv_chaos}.
\end{example}

\subsection{General McKean{--}Vlasov Setting}\label{scn:nn_applications}

\begin{example}[General Functionals]\label{ex:mfl}
    In the general setting, under Assumptions~\ref{as:functional_convexity},~\ref{as:functional_smoothness}, and~\ref{as:prox_gibbs_lsi}, the work of~\cite{suzuki2023convergence} provided the first discretization bounds. They impose further assumptions and their resulting complexity bound is rather complicated, but it reads roughly $MN = \widetilde{\mc O}(\nicefrac{\poly(\kappa)\, d^2}{\varepsilon^4})$ for the discretization of~\eqref{eq:finite_particle}.
    Subsequently,~\cite{fu2023mean} obtained an improved complexity of $MN = \widetilde{\mc O}(\nicefrac{\kappa^4 d^{3/2}}{\varepsilon^3})$ via ULMC in the averaged $\msf{TV}$ distance.
    In comparison, we can improve this complexity guarantee to $\widetilde{\mc O}(\nicefrac{\kappa^{5/2} d^{3/2}}{\varepsilon^3})$, and the guarantee even holds in $\sqrt{\KL}$ if we combine ULMC with the proximal sampler.
    It appears that we gain one factor of $\sqrt\kappa$ through sharper discretization analysis (via~\cite{zhang2023improved}, or via the error analysis of the proximal sampler in~\cite{altschuler2023faster}), and one factor of $\kappa$ via a sharper propagation of chaos result (Theorem~\ref{thm:entropic_prop_chaos_general_functional}).

    We also note that the result of~\cite{fu2023mean} is based on a kinetic version of the propagation of chaos argument from~\cite{Chen+24PoCKinetic}, whereas our approach uses the original ``non-kinetic'' argument from~\cite{chen2022uniform} in the form of Theorem~\ref{thm:entropic_prop_chaos_general_functional}.
\end{example}

\paragraph*{Application to Two-Layer Neural Networks.}
Let us consider the problem of learning a two-layer neural network in the mean-field regime. 
Let $f_\theta : \R^d\to\R$ be a function parameterized by $\theta \in \R^p$, and for any probability measure $\mu$ over $\R^p$, let $f_\mu \deq \int f_\theta \,\mu(\D\theta)$.
For example, in a standard two-layer neural network, we take $\theta = (a, w) \in \R\times \R^d$ and $f_{a,w}(x) = a\,\msf{ReLU}(\langle w,x\rangle)$.
When $\mu = \frac{1}{m} \sum_{j=1}^m \delta_{(a_j, w_j)}$ is an empirical measure, then $f_\mu$ is the function computed by a two-layer neural network with $m$ hidden neurons.
In this formulation, however, we can take $\mu$ to be any probability measure, corresponding to the \emph{mean-field limit} $m\to\infty$~\cite{chizat2018global, mei2018mean, Chi22MFL, RotVan22NNInteracting, sirignano2020mean}.

Given a dataset $\{(x_i, y_i)\}_{i=1}^n$ in $\R^d\times\R$ and a loss function $\ell : \R\times\R\to\R$, we can formulate neural network training as the problem of minimizing the loss $\mu \mapsto \sum_{i=1}^n \ell(f_\mu(x_i), y_i)$.
To place this within the general McKean{--}Vlasov framework, we add two regularization terms: 
(1) $\frac{\lambda}{2}\int \norm\cdot^2\,\D\mu$ 
corresponds to weight decay; 
and (2) $\frac{\sigma^2}{2}\int \log\mu \, \D \mu$ is entropic regularization. 
We are now in the setting of \S\ref{sec:mfl}, with $\mc F_0(\mu) = \sum_{i=1}^n \ell(f_\mu(x_i), y_i)$.

To minimize this energy, it is natural to consider the Euler--Maruyama discretization of~\eqref{eq:finite_particle}, which corresponds to learning the neural network via noisy GD\@, and was considered in~\cite{suzuki2023convergence}. Recent works \cite{fu2023mean, Chen+24PoCKinetic} also considered the underdamped version of \eqref{eq:mean_field} and its discretization.
Under the assumptions common to those works as well as our own, our results yield improved algorithmic guarantees for this task (see Example~\ref{ex:mfl}).

Unfortunately, the assumptions used for the analysis of the general McKean{--}Vlasov are restrictive and limit the applicability to neural network training.
For example, it suffices for $\ell$ to be convex in its first argument (to satisfy Assumption~\ref{as:functional_convexity}), to have two bounded derivatives (w.r.t.\ its first argument), and for $\theta \mapsto f_\theta(x_i)$ to have two bounded derivatives for each $x_i$.
The last condition is satisfied, e.g., for $f_\theta(x) = \tanh(\langle \theta, x\rangle)$.
For a genuinely two-layer example, we can take $f_\theta(x) = \tanh(a)\tanh(\langle w,x\rangle)$ for $\theta = (a,w) \in \R\times\R^d$.
Under these conditions, Assumptions~\ref{as:functional_smoothness} and~\ref{as:functional_bdd_grad} hold, which in turn furnish log-Sobolev inequalities via Lemmas~\ref{lem:prox_gibbs_lsi_informal} and~\ref{lem:lsiN_informal}.

\paragraph*{Limitations.} However, we note that there is a substantial limitation of our framework when applied to the mean-field Langevin dynamics.
Although we are able to establish a uniform-in-$N$ LSI for the stationary distribution $\mu^{1:N}$ under appropriate assumptions (see Corollary~\ref{cor:unif_LSI_crazy} for a precise statement), the dependence of the LSI constant scales poorly (in fact, doubly exponentially) in the problem parameters.
To fully benefit from the modularity of our approach, it is desirable to obtain a uniform-in-$N$ LSI with better scaling, and we leave this question open for future research.

\section{Conclusion}

In this work, we propose a framework for obtaining sampling guarantees for the minimizers of~\eqref{eq:total_energy_with_interaction} and~\eqref{eq:total_energy}, based on decoupling the problem into \emph{(i)} particle approximation via \emph{propagation of chaos}, and \emph{(ii)} time-discretization via \emph{log-concave sampling theory}. Our approach leads to simpler proofs and improved guarantees compared to previous works, and our results readily benefit from any improvements in either \emph{(i)} or \emph{(ii)}.

We conclude by listing some future directions of study.
As discussed in \S\ref{scn:nn_applications}, our uniform-in-$N$ LSI for the mean-field Langevin dynamics currently scales poorly in the problem parameters, and it is important to improve it.
We also believe there is further room for improvement in the propagation of chaos results. For example, can the sharp rate in Theorem~\ref{thm:N_choice_bias} be extended to stronger metrics such as R\'enyi divergences, as well as to situations when the weak interaction condition (Assumption~\ref{as:weak_interaction}) fails, e.g., in the strongly displacement convex case or in the setting of \S\ref{sec:mfl_chaos}?
For the sampling guarantees, the prior works~\cite{bou2023nonlinear, suzuki2023convergence} considered different settings, such as potentials satisfying convexity at infinity or the use of stochastic gradients; these extensions are compatible with our approach and could possibly lead to improvements in these cases, as well as others. Finally, consider the case where $\int \nabla W(X_t - \cdot) \, \D \pi_t$ in~\eqref{eq:mean_field_interactive} is replaced with a generic function $\int b_t(X_t, \cdot) \, \D \pi_t$, $b_t: \R^n \times \R^n \to \R^n$. It would be interesting to extend our analysis to this setting, as it arises in many applications~\cite{arnaudon2020second}.

\section*{Acknowledgements}
We thank Zhenjie Ren for pointing out an error in a previous draft of this paper, and Daniel Lacker, Atsushi Nitanda, and Taiji Suzuki for important discussions and references.
YK was supported in part by NSF awards CCF-2007443 and CCF-2134105.
MSZ was supported by NSERC through the CGS-D program. 
SC was supported by the Eric and Wendy
Schmidt Fund at the Institute for Advanced Study.
MAE was supported by NSERC Grant [2019-06167] and
CIFAR AI Chairs program at the Vector Institute.
MBL was supported by NSF grant DMS-2133806.

\newpage

\printbibliography

\newpage
\appendix

\section{Control of the Finite-Particle Error}\label{app:bias_control}

In this section, we prove the results in \S\ref{scn:bias_control} on the finite-particle error. We will make extensive use of the following transport inequality, which arises as a consequence of~\eqref{eq:lsi}.
\begin{lemma}[{Talagrand's Transport Inequality, \cite{otto2000generalization}}]
    If a measure $\pi$ satisfies \eqref{eq:lsi} with constant $C_{\msf{LSI}}$, then for all measures $\mu\in \Pac$,
    \begin{align}\label{eq:talagrand}\tag{TI}
        \mc W_2^2(\mu, \pi) \leq 2 C_{\msf{LSI}}\, \msf{KL}(\mu \mmid \pi)\,.
    \end{align}
\end{lemma}

\subsection{LSI Case}\label{scn:Entropic-POC-LSI}

We provide the proof of Theorem~\ref{thm:N_choice_bias} under the assumption of~\eqref{eq:lsi} for the invariant measures of $\eqref{eq:mean_field_interactive}$ and $\eqref{eq:finite_particle_interactive}$.
This relies on a BBGKY hierarchy based on the arguments of~\cite{lacker2023sharp}.

Recall that $\mu^{1:k}$ is the $k$-particle distribution of the finite-particle system.
Explicitly,
\begin{align*}
    \log \mu^{1:k}(x^{1:k})
    &= \log \int \exp\Bigl(- \frac{2}{\sigma^2} \sum_{i=1}^N V(x^i) - \frac{1}{\sigma^2\,(N-1)} \sum_{\substack{i,j=1 \\
     i\neq j}}^N W(x^i - x^j) \Bigr) \, \D x^{k+1:N} + \text{const.}
\end{align*}
Using exchangeability, we can then compute the gradient of the potential for this measure as
\begin{align*}
    &-\frac{\sigma^2}2\, \nabla_{x^i}\log \mu^{1:k}(x^{1:k}) \\
    &\qquad = \nabla V(x^i) + \frac{1}{N-1} \sum_{\substack{j=1 \\ i \neq j}}^k \nabla W(x^i - x^j)
    + \frac{N-k}{N-1} \E_{\mu^{k+1|1:k}(\cdot \mid x^{1:k})} \nabla W(x^i - \cdot)\,.
\end{align*}
Let $X^{1:k} \sim \mu^{1:k}$ and introduce the notation
\begin{align*}
    \msf{K}_k \defeq \msf{KL}(\mu^{1:k} \mmid \pi^{\otimes k})\,.
\end{align*}
Invoking \eqref{eq:lsi} of the mean-field invariant measure (and tensorizing) leads to
\begin{align*}
    \msf{K}_k
    &\leq \frac{C_{\msf{LSI}}(\pi)}{2}\, \msf{FI}(\mu^{1:k} \mmid \pi^{\otimes k}) \\
    &= \frac{2C_{\msf{LSI}}(\pi)}{\sigma^4} \sum_{i=1}^k \E\Bigl[\Bigl\lVert\frac{1}{N-1}\sum_{\substack{j=1 \\ j \neq i}}^k \nabla W(X^i - X^j) - \int \nabla W(X^i - \cdot) \,\D \pi \\
    &\qquad\qquad\qquad \qquad \qquad \qquad{} + \frac{N-k}{N-1} \int \nabla W(X^i - \cdot) \,\D \mu^{k+1|1:k}(\cdot \mid X^{1:k}) \Bigr\rVert^2\Bigr] \\
    &\leq \frac{4k\,C_{\msf{LSI}}(\pi)}{\sigma^4\,{(N-1)}^2}\,  \underset{\msf A}{\underbrace{\E\Bigl[\Bigl\lVert\sum_{\substack{j=2}}^k\Bigl( \nabla W(X^1 - X^j) - \int \nabla W(X^1 - \cdot)\, \D \pi\Bigr)\Bigr\rVert^2\Bigr]}}\\
    &\qquad{} + \frac{4k\,C_{\msf{LSI}}(\pi)\,{(N-k)}^2}{\sigma^4\,{(N-1)}^2} \, \underset{\msf B}{\underbrace{\E\Bigl[\Bigl\lVert\int \nabla W(X^1 - \cdot) \,\bigl(\D \mu^{k+1|1:k}(\cdot \mid X^{1:k}) - \D \pi\bigr)\Bigr\rVert^2\Bigr]}}\,,
\end{align*}
where the last line follows from exchangeability and $\norm{a+b}^2 \leq 2\,(\norm{a}^2 +\norm{b}^2)$ for vectors $a,b\in\R^d$. 

\subsubsection{Bounding the Error Terms}
We now handle terms $\msf A, \msf B$ separately.
\begin{align*}
    \msf{A} 
    &= \sum_{j=2}^k \E[\norm{\nabla W(X^1 - X^j) - \E_\pi\nabla W(X^1 - \cdot)}^2] \\
    &\qquad + \sum_{\substack{i,j=2 \\i \neq j}}^k \E \bigl\langle \nabla W(X^1 - X^i) - \E_\pi \nabla W(X^1 - \cdot),\,\nabla W(X^1 - X^j) - \E_{\pi}\nabla W(X^1 - \cdot)\bigr\rangle \\
    &\underset{\text{(i)}}{=} (k-1) \E[\norm{\nabla W(X^1 - X^2) - \E_{ \pi}\nabla W(X^1 - \cdot)}^2] + \\
    &\qquad +(k-1)\,(k-2)\E\bigl \langle \nabla W(X^1- X^2) - \E_{\pi} \nabla W(X^1 - \cdot),\\
    &\qquad\qquad\qquad\qquad\qquad \qquad\qquad\qquad\qquad\qquad\qquad\nabla W(X^1 - X^3) - \E_{\pi}\nabla W(X^1 - \cdot)\bigr\rangle \\
    &\underset{\text{(ii)}}{\leq} (k-1)\, \beta_W^2 \E[\norm{X - Y}^2] \\
    &\quad + (k-1)^2 \E\bigl \langle \nabla W(X^1- X^2) - \E_{\pi} \nabla W(X^1 - \cdot),\,
    \nabla W(X^1 - X^3) - \E_{\pi}\nabla W(X^1 - \cdot)\bigr\rangle\,,
\end{align*}
where we used the exchangeability of the particles in (i) and the smoothness of $W$ in (ii).
Here, $X\sim \mu^1$ and $Y\sim \pi$ are independent.

Let us deal with these two terms separately.
For the first term, let $\bar Y \sim \pi$ be \emph{optimally} coupled with $X$. Then, by independence and sub-Gaussian concentration (implied by~\eqref{eq:lsi}),
\begin{align}
    \E[\norm{X-Y}^2]
    &\le 2\E[\norm{X-\bar Y}^2] + 2\E[\norm{Y-\bar Y}^2]
    = 2\,\mc W_2^2(\mu^1,\pi) + 4\E[\norm{Y-\E Y}^2] \nonumber\\
    &\le 4C_{\msf{LSI}}(\pi) \KL(\mu^1 \mmid \pi) + 4d C_{\msf{LSI}}(\pi)
    \le 4 C_{\msf{LSI}}(\pi)\, (\msf{K}_3 + d)\,,\label{eq:second_moment_bd}
\end{align}
where the second inequality follows from \eqref{eq:talagrand}, and the last one follows from the data-processing inequality for the $\KL$ divergence.
For the second term, the Cauchy--Schwarz inequality leads to
\begin{align}
    &\E\bigl \langle \nabla W(X^1- X^2) - \E_{\pi} \nabla W(X^1 - \cdot),\, \nabla W(X^1 - X^3) - \E_{\pi}\nabla W(X^1 - \cdot)\bigr\rangle \nonumber\\
    &\qquad = \E\bigl \langle \nabla W(X^1- X^2) - \E_{\pi} \nabla W(X^1 - \cdot),\,\E_{\mu^{3 \mid 1:2}(\cdot \mid X^{1:2})} \nabla W(X^1 - \cdot) - \E_{\pi}\nabla W(X^1 - \cdot)\bigr\rangle \nonumber\\
 & \qquad\leq\beta_{W}^{2}\sqrt{\E[\norm{X-Y}^{2}]}\sqrt{\E \mc W_2^2\bigl(\mu^{3\mid 1:2}(\cdot \mid X^{1:2}),\, \pi\bigr)}\nonumber \\
 & \qquad\underset{\text{(i)}}{\leq}\beta_{W}^{2}\sqrt{4C_{\msf{LSI}}(\pi)\,(\msf K_{3}+d)}\sqrt{2C_{\msf{LSI}}(\pi)\E\KL\bigl(\mu^{3\mid 1:2}(\cdot \mid X^{1:2}) \bigm\Vert \pi\bigr)}\nonumber \\
 & \qquad\underset{\text{(ii)}}{\leq}3\beta_{W}^{2}C_{\msf{LSI}}(\pi)\sqrt{\msf K_{3}+d}\sqrt{\msf K_{3}}\label{eq:boosting-kN}\\
 & \qquad\leq3\beta_{W}^{2}C_{\msf{LSI}}(\pi)\,(\msf K_{3}+d)\,,\nonumber 
\end{align}
where in (i) we applied the bound~\eqref{eq:second_moment_bd} as well as~\eqref{eq:talagrand}, and in (ii) we used the chain rule for the $\msf{KL}$ divergence.

We return to the analysis of the term $\msf B$. In a similar way,
we obtain 
\begin{align*}
\msf B
&= \E\Bigl[\Bigl\lVert\int \nabla W(X^1 - \cdot) \,\bigl(\D \mu^{k+1|1:k}(\cdot \mid X^{1:k}) - \D \pi\bigr)\Bigr\rVert^2\Bigr]
\le \beta_{W}^{2}\E\mc W_{2}^{2}\bigl(\mu^{k+1|1:k}(\cdot\mid X^{1:k}),\,\pi\bigr)\\
 & \leq2\beta_{W}^{2}C_{\msf{LSI}}(\pi)\,(\msf K_{k+1}-\msf K_{k})\,.
\end{align*}

\subsubsection{Induction}

Putting our bounds on $\msf A$ and $\msf B$ together, we obtain for $N\geq 30$,
\begin{equation}
\msf K_{k}\leq\frac{30 k^{3}\beta_{W}^{2}C_{\msf{LSI}}^{2}(\pi)}{\sigma^{4} N^2}\,(\msf K_{3}+d)+\frac{8k\beta_{W}^{2}C_{\msf{LSI}}^{2}(\pi)}{\sigma^{4}}\,(\msf{K_{k+1}-\msf K_{k})}\,.\label{eq:recursive-ineq}
\end{equation}
In particular, the case of $k=N$ involves our bounds only on $\msf A$,
leading to
\[
\msf K_{N}\le\frac{30N\beta_{W}^{2}C_{\msf{LSI}}^{2}(\pi)}{\sigma^{4}}\,(\msf K_{3}+d)\,.
\]
By grouping together the $\msf{K}_k$ terms in~\eqref{eq:recursive-ineq}, 
\begin{align}
\msf K_{k}\leq\underset{=:\mc C_{k}}{\underbrace{\frac{8k\beta_{W}^{2}C_{\msf{LSI}}^{2}(\pi)/\sigma^{4}}{1+8k\beta_{W}^{2}C_{\msf{LSI}}^{2}(\pi)/\sigma^{4}}}}\,\Bigl(\msf K_{k+1}+\bigl(\frac{2k}{N}\bigr)^{2}\,(\msf K_{3}+d)\Bigr)\,.\label{eq:recursion-on-Kk}
\end{align}
Iterating this inequality down to $k=3$, for $\rho \defeq \nicefrac{\sigma^{4}}{8\beta_{W}^{2}C_{\msf{LSI}}^{2}(\pi)}$,
\begin{align*}
\msf K_{3} & \leq \Bigl(\prod_{k=3}^{N-1}\mc C_{k}\Bigr)\,\frac{30N\beta_{W}^{2}C_{\msf{LSI}}^{2}(\pi)}{\sigma^{4}}\,(\msf K_{3}+d)+\sum_{k=3}^{N-1}\Bigl(\prod_{\ell=3}^{k}\mc C_{\ell}\Bigr)\,\bigl(\frac{2k}{N}\bigr)^{2}\,(\msf K_{3}+d)\\
 &\le \underbrace{\biggl[\Bigl(\prod_{k=3}^{N-1}\mc C_{k}\Bigr)\,\frac{4N}{\rho}+\sum_{k=3}^{N-1}\Bigl(\prod_{\ell=3}^{k}\mc C_{\ell}\Bigr)\,\bigl(\frac{2k}{N}\bigr)^{2}\biggr]}_{\eqqcolon c_{N}}\,(\msf K_{3}+d)\,.
\end{align*}

Now we show $c_{N}<1/2$, which implies $\msf K_{3}\leq2c_{N}d$.
We require the following lemma.

\begin{lemma}\label{lem:coefficient_product_bound}
For $3\leq i\leq k\leq N$,
\[
\prod_{\ell=i}^{k}\mc C_{\ell}\leq\Bigl(\frac{i+\rho}{k+1+\rho}\Bigr)^{\rho}\,.
\]
\end{lemma}
\begin{proof}
For $\mc C_{\ell}=\frac{\ell\rho^{-1}}{1+\ell\rho^{-1}}$, we have
\[
C\deq\log\prod_{\ell=i}^{k}\mc C_{\ell}=
\sum_{\ell=i}^{k}\log\Bigl(1-\frac{1}{1+\ell\rho^{-1}}\Bigr)
\leq-\sum_{\ell=i}^{k}\frac{1}{1+\ell\rho^{-1}}\,.
\]
As the summand is decreasing in $\ell$, it follows that 
\[
C\leq
- \sum_{\ell=i}^k \int_\ell^{\ell+1} \frac{1}{1+x\rho^{-1}} \, \D x
= -\int_{i}^{k+1}\frac{1}{1+x\rho^{-1}}\,\D x=-\rho\log\frac{k+1+\rho}{i+\rho}\,.
\]
Therefore, 
\begin{align*}
\prod_{\ell=i}^{k}\mc C_{\ell}=\exp C\leq \Bigl(\frac{k+1+\rho}{i+\rho}\Bigr)^{-\rho}\,,
\end{align*}
which proves the lemma.
\end{proof}

Using Lemma~\ref{lem:coefficient_product_bound}, we obtain 
\[
c_{N}\leq 4\,{(3+\rho)}^{\rho}\,\Bigl(\frac{N^{1-\rho}}{\rho}+\frac{1}{N^{2}}\sum_{k=3}^{N-1}k^{2-\rho}\Bigr)\,.
\]
Under Assumption~\ref{as:weak_interaction}, i.e. $\rho \geq 3$, we may assume $\rho = 3$ since we can always take a worse bound on the constants $\beta_{W}$ so that $\rho=3$.
As seen shortly, the rate does not improve even if $\rho > 3$.\footnote{Alternatively, one can show the bound in Lemma~\ref{lem:coefficient_product_bound} decreases in $\rho$, so we can just substitute $\rho = 3$ therein.}
For $\rho = 3$ and $N\ge 100$, we therefore obtain
\[
c_{N}\leq 864\, \Bigl(\frac{1}{3N^{2}}+\frac{1}{N^{2}} \sum_{k=3}^{N-1} \frac{1}{k}\Bigr)\leq \frac 1 2\,,
\]
and thus
\begin{align}\label{eq:K3-bound}
\msf K_{3}\lesssim\frac{d\log N}{N^{2}}\,.
\end{align}

\subsubsection{Bootstrapping} 
Substituting the bound \eqref{eq:K3-bound} for $\msf K_{3}$ into the recursive inequality  \eqref{eq:recursion-on-Kk}, we end up with a suboptimal rate of $\widetilde{\mc O}(k^{3}/N^{2})$ for $\msf K_{k}$. 
To improve the bound, we substitute
our established bound (\ref{eq:K3-bound}) into (\ref{eq:boosting-kN}), which results in an improved recursive inequality.
Indeed,
\begin{align*}
    \msf A
    &\lesssim k\beta_W^2 C^2_{\msf{LSI}}(\pi)\,({\KL_3} + d) + k^2 \beta_W^2 C_{\msf{LSI}}(\pi)\sqrt{{\KL_3} + d}\sqrt{\KL_3}
    \lesssim dk\beta_W^2 C_{\msf{LSI}}(\pi)\sqrt{\log N}
\end{align*}
and therefore
\begin{align*}
\msf K_{k} &\leq \widetilde{\mc O}\Bigl( \frac{dk^2 \beta_W^2 C_{\msf{LSI}}^2(\pi)}{\sigma^4 N^2}\Bigr)
+\frac{8k\beta_{W}^{2}C_{\msf{LSI}}^2(\pi)}{\sigma^{4}}\,(\msf K_{k+1}-\msf K_{k})\,.
\end{align*}
For $k=N$ this yields
\begin{align*}
\msf K_{N} &\le \widetilde{\mc O}\Bigl( \frac{d\beta_W^2 C_{\msf{LSI}}^2(\pi)}{\sigma^4}\Bigr)\,.
\end{align*}
Regrouping $\msf K_{k}$ as before, we obtain
\[
\msf K_{k}\leq\mc C_{k}\,\Bigl(\msf K_{k+1}+\widetilde{\mc O}\bigl( \frac{dk}{N^2}\bigr)\Bigr)\,.
\]
Iterating this down to $k=N$,
\begin{align*}
\msf K_{k}
 &\leq\Bigl(\prod_{\ell=k}^{N-1}\mc C_\ell\Bigr)\,\msf K_{N}+\sum_{\ell=k}^{N-1}\Bigl(\prod_{j=k}^{\ell}\mc C_{j}\Bigr) \,\widetilde{\mc O}\bigl(\frac{d\ell}{N^2}\bigr) \\
 & \underset{\text{(i)}}{\le} \widetilde{\mc O}\Bigl(\frac{k^3}{N^3} \, \frac{d\beta_W^2 C_{\msf{LSI}}^2(\pi)}{\sigma^4} + \sum_{\ell=k}^{N-1} \frac{k^3}{\ell^3} \, \frac{d\ell}{N^2} \Bigr)
 \underset{\text{(ii)}}{\le} \widetilde{\mc O}\bigl(\frac{dk^2}{N^2}\bigr)\,,
\end{align*}
where in (i) we used Lemma~\ref{lem:coefficient_product_bound} with
$\rho=3$, and (ii) follows from $\rho \ge 3$ and $\sum_{\ell\geq k} \ell^{-2}\leq k^{-1}$.
Therefore, for some fixed $k$ it suffices to take  $N=100\vee\widetilde{\Omega}(k\sqrt d/\epsilon)$ to achieve $\epsilon^2$-bias in $\msf{KL}$, completing the proof of Theorem~\ref{thm:N_choice_bias}.

\subsection{Strongly Convex Case}

The following propagation of chaos argument for the strongly log-concave case is based on~\cite{sznitman1991topics}. Let $(X^{1:N}_t)_{t \geq 0}$ denote the stochastic process following the finite-particle stochastic differential equation~\eqref{eq:finite_particle_interactive}.
Let the corresponding semigroup be denoted ${(\mc T_t)}_{t\ge 0}$, defined as follows.
For any test function $f: \R^{d \times N} \to \R$,
\begin{align*}
    \mc T_t f(x^{1:N}) = \E[f(X^{1:N}_t)\mid X_0^{1:N} = x^{1:N}]\,.
\end{align*}
Then, the following simple lemma proves Wasserstein contraction for the finite-particle system.

\begin{lemma}\label{lem:one_step_w2_contract}
    Under Assumption~\ref{as:str_cvx_VW} and for $N \ge \frac{\alpha_V - {\alpha_W^-}}{\alpha_V + {(\alpha_W)}_-}$, ${(\mc T_t)}_{t\ge 0}$ is a contraction in the $2$-Wasserstein distance with exponential rate at least $\alpha/2$, where $\alpha \deq \alpha_V + {\alpha_W^-}$.
    In other words, for any measures $\mu_0^{1:N}$, $\nu_0^{1:N}$ in $\mc P_2(\R^{d\times N})$,
    \begin{align*}
        \mc W_2(\mu_0^{1:N} \mc T_t,\, \nu_0^{1:N} \mc T_t)
        &\leq \exp(-\alpha t/2)\, \mc W_2(\mu_0^{1:N}, \nu_0^{1:N})\,.
    \end{align*}
\end{lemma}

\begin{proof}
    Note that ${(\mc T_t)}_{t\ge 0}$ corresponds to the time-scaled (by factor $\sigma^2/2$) Langevin diffusion with stationary distribution $\mu^{1:N}$, which is $\frac{2}{\sigma^2}\,(\alpha_V + \frac{N}{N-1}\,{\alpha_W^-})$-strongly log-concave by Lemma~\ref{lem:strongCvxCase}. The condition on $N$ ensures that this is at least $\alpha/\sigma^2$. Consequently, it is well-known (e.g., via synchronous coupling) that the diffusion is a contraction in the Wasserstein distance with rate at least $\alpha/2$.
\end{proof}

We next bound the error incurred in one step from applying the finite-particle semigroup to $\pi^{\otimes N}$.

\begin{lemma}\label{lem:one_step_w2_err}
    Under Assumptions~\ref{as:smoothness} and~\ref{as:str_cvx_VW},
    for any $\lambda > 0$,
    $\mc T_{h}$ induces the following error in Wasserstein distance:
    \begin{align*}
        \mc W_2^2(\pi^{\otimes N} \mc T_h, \, \pi^{\otimes N})
        &\le \frac{(1+\lambda^{-1})\,\beta_W^2 \sigma^2 d h^2}{\alpha} \exp\Bigl(\frac{(1+\lambda)\,\beta_W^2 h^2}{2}\Bigr)\,.
    \end{align*}
\end{lemma}

\begin{proof}
    We resort to a coupling argument, noting that $\pi$ is stationary under \eqref{eq:mean_field_interactive}. 
    Starting with $\pi^{\otimes N}$, we evolve ${(X^{1:N}_t)}_{t\ge 0}$ and ${(Y^{1:N}_t)}_{t\ge 0}$ according to \eqref{eq:finite_particle_interactive} and \eqref{eq:mean_field_interactive} respectively, i.e. $X^{1:N}_t \sim \pi^{\otimes N}\mc T_t$ and $Y^{1:N}_t \sim \pi^{\otimes N}$. This argument is adapted from the original propagation of chaos proof by \cite{sznitman1991topics}.

We can compute the evolution under a synchronous coupling as:
\begin{align*}
    \D (X^i_t - Y^i_t)
    &=  -\bigl(\nabla V(X^i_t) -\nabla V(Y^i_t) \bigr)\, \D t - \frac{1}{N-1} \sum_{\substack{j=1 \\ j \neq i}}^N \bigl(\nabla W(X^i_t - X^j_t) - \E_{\pi}\nabla W(Y^i_t - \cdot) \bigr)\, \D t \\
    &= -\bigl(\nabla V(X^i_t) -\nabla V(Y^i_t) \bigr)\, \D t - \frac{1}{N-1} \sum_{\substack{j=1 \\ j \neq i}}^N \bigl(\nabla W(X^i_t - X^j_t) - \nabla W(Y^i_t - X^j_t) \bigr)\, \D t \\
    &\qquad{} - \frac{1}{N-1} \sum_{\substack{j=1 \\ j \neq i}}^N \bigl(\nabla W(Y^i_t - X^j_t) - \nabla W(Y^i_t - Y^j_t)  \bigr)\, \D t \\
    &\qquad{} - \frac{1}{N-1} \sum_{\substack{j=1 \\ j \neq i}}^N \bigl(\nabla W(Y^i_t - Y^j_t) - \E_{\pi}\nabla W(Y^i_t - \cdot)  \bigr)\, \D t\,.
\end{align*}

Now let us denote by $\gradW(x, y) \deq \nabla W(x-y) - \E_{\pi} \nabla W(x-\cdot)$ the centered gradient (with respect to $\pi$).
By It\^o's formula and Assumption~\ref{as:str_cvx_VW},
\begin{align*}
    \D \norm{X_t^i - Y_t^i}^2
    &= 2\,\langle X_t^i - Y_t^i, \D (X_t^i - Y_t^i)\rangle \\
    &\le -2\,(\alpha_V + \alpha_W)\,\norm{X_t^i - Y_t^i}^2 \, \D t \\
    &\qquad{} -  \frac{2}{N-1} \sum_{\substack{j=1 \\ j \neq i}}^N \langle X_t^i - Y_t^i, \nabla W(Y^i_t - X^j_t) - \nabla W(Y^i_t - Y^j_t) \rangle \, \D t \\
    &\qquad{} -  \frac{2}{N-1} \sum_{\substack{j=1 \\ j \neq i}}^N \langle X_t^i - Y_t^i, \nabla W(Y^i_t - Y^j_t) - \E_{\pi}\nabla W(Y^i_t - \cdot) \rangle \, \D t \\
    &\le \frac{2\beta_W\,\norm{X_t^i - Y_t^i}}{N-1}\sum_{\substack{j=1 \\ j \neq i}}^N {\norm{X_t^j - Y_t^j}}\,\D t + \frac{2\,\norm{X_t^i - Y_t^i}}{N-1} \,\Bigl\lVert\sum_{\substack{j=1 \\ j \neq i}}^N\gradW(Y_t^i, Y_t^j)\Bigr\rVert \, \D t
\end{align*}
or
\begin{align*}
    \D \norm{X_t^i - Y_t^i}
    &\le \frac{\beta_W}{N-1}\sum_{\substack{j=1 \\ j \neq i}}^N {\norm{X_t^j - Y_t^j}}\,\D t + \frac{1}{N-1} \,\Bigl\lVert\sum_{\substack{j=1 \\ j \neq i}}^N\gradW(Y_t^i, Y_t^j)\Bigr\rVert \, \D t\,.
\end{align*}
Integrating and squaring,
\begin{align*}
    \norm{X_t^i - Y_t^i}^2
    &\le \Bigl\lvert \int_0^t \Bigl( \frac{\beta_W}{N-1}\sum_{\substack{j=1 \\ j \neq i}}^N {\norm{X_s^j - Y_s^j}} + \frac{1}{N-1} \,\Bigl\lVert\sum_{\substack{j=1 \\ j \neq i}}^N\gradW(Y_s^i, Y_s^j)\Bigr\rVert\Bigr) \, \D s\Bigr\rvert^2 \\
    &\le \frac{(1+\lambda)\,\beta_W^2 t}{N-1}\sum_{\substack{j=1 \\ j \neq i}}^N \int_0^t \norm{X_s^j - Y_s^j}^2 \, \D s + \frac{(1+\lambda^{-1})\,t}{{(N-1)}^2} \int_0^t \Bigl\lVert\sum_{\substack{j=1 \\ j \neq i}}^N\gradW(Y_s^i, Y_s^j)\Bigr\rVert^2 \, \D s\,,
\end{align*}
where the last line follows from Young's inequality.

Next, we take expectations.
Note that $\gradW(\cdot, \cdot)$ is centered in its second variable, so for any $j \neq k$,
\begin{align*}
    \E\langle\gradW(Y^i_t, Y^j_t), \gradW(Y^i_t, Y^k_{t})\rangle = 0\,.
\end{align*}
Otherwise, we can bound the terms via
\begin{align*}
    \E[\norm{\gradW(Y^i_t, Y^j_t)}^2]
    \leq \beta_W^2\E_{\substack{Y^j_t\sim \pi\\ Z \sim \pi}} [\norm{Y^j_t - Z}^2]
    \leq \frac{\beta_W^2 \sigma^2 d}{\alpha}\,.
\end{align*}
Here, $Z$ is an independent draw from $\pi$ and so cannot be reduced via coupling. The second inequality follows from a standard bound on the centered second moment of a strongly log-concave measure, using the fact that $\pi$ is $2\alpha/\sigma^2$-strongly log-concave~\cite[c.f.][]{DalKarRio22NonStrongly}.

Therefore, taking expectations and summing over the particles,
\begin{align*}
    \E[\norm{X_t^{1:N} - Y_t^{1:N}}^2]
    &\le (1+\lambda)\,\beta_W^2 t \int_0^t \norm{X_s^{1:N} - Y_s^{1:N}}^2 \, \D s + \frac{(1+\lambda^{-1})\,\beta_W^2 \sigma^2 d t^2}{\alpha}\,.
\end{align*}
By Gr\"onwall's inequality below,
\begin{align*}
    \E[\norm{X_h^{1:N} - Y_h^{1:N}}^2]
    &\le \frac{(1+\lambda^{-1})\,\beta_W^2 \sigma^2 d h^2}{\alpha} \exp\Bigl(\frac{(1+\lambda)\,\beta_W^2 h^2}{2}\Bigr)\,.
\end{align*}
This concludes the proof.
\end{proof}

\begin{lemma}[Gr\"onwall's Inequality]\label{lem:Gronwall}
    For $T>0$, let $f:[0,T] \to \R_{\geq 0}$ be bounded. Suppose that the following holds pointwise for some functions $a, b : [0,T] \to \R$, where $a$ is increasing:
    \[
    f(t) \leq a(t) + \int_0^t b(s) f(s)\, \D s\,.
    \]
    Then,
    \[
    f(t) \leq a(t) \exp\Bigl(\int_0^t b(s) \, \D s\Bigr)\,.
    \]
\end{lemma}

Composing Lemmas~\ref{lem:one_step_w2_contract} and~\ref{lem:one_step_w2_err}, we now prove our propagation of chaos results.

\medskip{}

\begin{proofof}{Theorem~\ref{thm:weak_prop_chaos}}
    Indeed, we have
    \begin{align*}
        &\mc W_2(\mu^{1:N}, \pi^{\otimes N})
        = \mc W_2(\mu^{1:N} \mc T_h,\, \pi^{\otimes N})
        \le \mc W_2(\mu^{1:N} \mc T_h,\, \pi^{\otimes N} \mc T_h) + \mc W_2(\pi^{\otimes N} \mc T_h, \, \pi^{\otimes N}) \\
        &\qquad\le \exp(-\alpha h/2)\,\mc W_2(\mu^{1:N}, \pi^{\otimes N}) + \sqrt{\frac{(1+\lambda^{-1})\,\beta_W^2 \sigma^2 d h^2}{\alpha}}\exp\Bigl(\frac{(1+\lambda)\,\beta_W^2 h^2}{4}\Bigr)\,.
    \end{align*}
    Rearranging,
    \begin{align*}
        \mc W_2(\mu^{1:N}, \pi^{\otimes N})
        &\le \frac{1}{1-\exp(-\alpha h/2)}\sqrt{\frac{(1+\lambda^{-1})\,\beta_W^2 \sigma^2 d h^2}{\alpha}}\exp\Bigl(\frac{(1+\lambda)\,\beta_W^2 h^2}{4}\Bigr)\,.
    \end{align*}
    Let $h\searrow 0$ first and then $\lambda \nearrow \infty$ to obtain
    \begin{align*}
        \mc W_2^2(\mu^{1:N}, \pi^{\otimes N})
        &\le \frac{4\beta_W^2 \sigma^2 d}{\alpha^3}\,.
    \end{align*}
    Finally, when $k < N$, we use exchangeability (see Lemma~\ref{lem:wass_exchangeable} below) to conclude the proof of~\eqref{eq:w2_weak_prop}.

For~\eqref{eq:kl_prop}, by the Bakry--\'Emery condition we have $C_{\msf{LSI}}(\pi) \leq \nicefrac{\sigma^2}{2\alpha}$, and tensorization~\cite[c.f.][Proposition 5.2.7]{bakry2014analysis} leads to $C_{\msf{LSI}}(\pi^{\otimes N}) \leq \nicefrac{\sigma^2}{2\alpha}$.
Thus, \eqref{eq:talagrand} leads to 
\[
    \msf{KL}(\mu^{1:N} \mmid \pi^{\otimes N}) \leq \frac{\sigma^2}{4\alpha}\, \msf{FI}(\mu^{1:N} \mmid \pi^{\otimes N})\,.
\]

However, one notes that the density of $\mu^{1:N}$ is log-smooth with parameter $\frac{2}{\sigma^2}\,(\beta_V + \frac{N}{N-1}\, \beta_W)$ (Lemma~\ref{lem:strongCvxCase}). Likewise, $\pi^{\otimes N}$ is log-smooth with parameter $\frac{2}{\sigma^2}\,(\beta_V + \beta_W)$. Now consider a functional $\mc F$ on the space of probability measures on $\mc P_{2,\rm ac}(\R^{d\times N})$ given by $\mc F: \nu \mapsto \E_{\nu}[\norm{\nabla \log \frac{\mu^{1:N}}{\pi^{\otimes N}}}^2]$.
Note that $\log(\mu^{1:N}/\pi^{\otimes N})$ is smooth with parameter at most $\frac{4}{\sigma^2}\,(\beta_V + \frac{N}{N-1}\, \beta_W) \le \frac{8}{\sigma^2}\,(\beta_V + \beta_W)$, for $N \ge 2$.

Next, note that for $Y^{1:N}\sim \pi^{\otimes N}$,
\begin{align*}
    \mc F(\pi^{\otimes N}) &= \E_{\pi^{\otimes N}} [\norm{\nabla \log \mu^{1:N} - \nabla \log \pi^{\otimes N}}^2] \\
    &= \frac{4N}{\sigma^4\,{(N-1)}^2} \E \Bigl[\Bigl\lVert\sum_{j=2}^N \Bigl(\nabla W(Y^1 - Y^j) - \int \nabla W(Y^1 - \cdot)\, \D \pi\Bigr)\Bigr\rVert^2\Bigr] \\
    &= \frac{4N}{\sigma^4\,{(N-1)}^2} \E\Bigl[\Bigl\lVert\sum_{j=2}^N \gradW(Y^1, Y^j)\Bigr\rVert^2\Bigr]\,,
\end{align*}
by using exchangeability and the definition of $\gradW$.

Subsequently, one derives the following inequality using the Wasserstein distance bound:
\begin{align*}
    \mc F(\mu^{1:N}) &\leq \frac{128}{\sigma^4}\,{(\beta_V + \beta_W)}^2\, \mc W_2^2(\mu^{1:N}, \pi^{\otimes N}) + 2\mc F(\pi^{\otimes N}) \\
    &\leq \frac{512 \beta_W^2 d}{\alpha^3\sigma^2}\, {(\beta_V + \beta_W)}^2 + \frac{8N}{\sigma^4\,{(N-1)}^2}\E\Bigl[\Bigl\lVert\sum_{j=2}^N \gradW(Y^1, Y^j)\Bigr\rVert^2\Bigr] \\
    &\le \frac{512 \beta_W^2 d}{\alpha^3 \sigma^2}\, {(\beta_V + \beta_W)}^2 + \frac{16\beta_W^2 N}{\sigma^4\,(N-1)}\E[\norm{Y^1-\E Y^1}^2] \\
    &\le \frac{512 \beta_W^2 d}{\alpha^3 \sigma^2}\, {(\beta_V + \beta_W)}^2 + \frac{16\beta_W^2 d}{\alpha \sigma^2}\,,
\end{align*}
by using \eqref{eq:w2_weak_prop} and the fact that $\gradW(\cdot,\cdot)$ is a centered random variable in its second argument. This concludes the proof for $k=N$, and as in the $\mc W_2^2$ bound, Lemma~\ref{lem:information_ineq} will conclude the proof for $k < N$.
\end{proofof}

\subsection{General Functional Case}\label{ap:general_functional_poc}

For any measure $\mu$, define its entropy as $\ent(\mu) = \int \log \mu \, \D \mu$.
We now provide a self-contained propagation of chaos argument in the general McKean{--}Vlasov setting, following~\cite{chen2022uniform}.
We begin with the following entropy toast inequality, i.e. half of the entropy sandwich inequality from~\cite{chen2022uniform}.

\begin{lemma}[Entropy Toast Inequality]\label{lem:entropy_toast}
    Define the empirical total energy for an $N$-finite particle system as follows.
    Given a measure $\nu^{1:N} \in \mc P_{2, \rm{ac}}(\R^{d \times N})$,
    \begin{align*}
        \mc E^N(\nu^{1:N}) = N \int \mc F(\rho_{x^{1:N}}) \, \nu^{1:N}(\D x^{1:N}) + \frac{\sigma^2}{2} \ent(\nu^{1:N})\,.
    \end{align*}

    Under Assumption~\ref{as:functional_convexity}, it holds for all measures $\nu^{1:N} \in \mc P_{2,\rm ac}(\R^{d \times N})$
    \[
        \frac{\sigma^2}{2} \KL(\nu^{1:N} \mmid \pi^{\otimes N}) \leq \mc E^N(\nu^{1:N}) - N \mc E(\pi)\,,
    \]
    where $\mc E$ is the total energy~\eqref{eq:total_energy} and $\pi$ is the stationary measure~\eqref{eq:mean_field_stat}.
\end{lemma}

\begin{proof}
    By Assumption~\ref{as:functional_convexity}, we have
    \begin{align*}
        &\mc E^N(\nu^{1:N}) - N\mc E(\pi)
        = N \E_{x^{1:N} \sim \nu^{1:N}}[\mc F(\rho_{x^{1:N}}) - \mc F(\pi)] + \frac{\sigma^2}{2}\,\bigl(\ent(\nu^{1:N}) - N\ent(\pi) \bigr) \\
        &\qquad \geq \E_{x^{1:N} \sim \nu^{1:N}}\Bigl[N \int \delta \mc F(\pi, z)\,(\rho_{x^{1:N}}(\D z) - \pi(\D z))\Bigr] + \frac{\sigma^2}{2}\,\bigl(\ent(\nu^{1:N}) - N\ent(\pi) \bigr) \\
        &\qquad =-\frac{\sigma^2}{2}\E_{x^{1:N} \sim \nu^{1:N}}\Bigl[N \int \log \pi(z)\,(\rho_{x^{1:N}}(\D z) - \pi(\D z))\Bigr] + \frac{\sigma^2}{2}\,\bigl(\ent(\nu^{1:N}) - N\ent(\pi) \bigr) \\
        &\qquad = -\frac{\sigma^2}{2} \E_{x^{1:N} \sim \nu^{1:N}}\Bigl[N \int \log \pi(z)\,\rho_{x^{1:N}}(\D z)\Bigr] + \frac{\sigma^2}{2}\ent(\nu^{1:N}) \\
        &\qquad = -\frac{\sigma^2}{2} \int \sum_{i=1}^N \log \pi(x^i) \, \nu^{1:N}(\D x^{1:N}) + \frac{\sigma^2}{2} \ent(\nu^{1:N})\,. 
    \end{align*}
    However, this is just $\frac{\sigma^2}{2} \KL(\nu^{1:N} \mmid \pi^{\otimes N})$, so we are done.
\end{proof}

\begin{proofof}{Theorem~\ref{thm:entropic_prop_chaos_general_functional}}
    We bound $\mc E^N(\mu^{1:N}) - N\mc E^N(\pi)$ via the following argument. First, define the finite-particle mean-field functional as $\mc F^N(\nu^{1:N}) = N \int \mc F(\rho_{x^{1:N}}) \, \nu^{1:N}(\D x^{1:N})$. In the sequel, we also use the following notation for conditional measures: if $x^{-i} \deq  (x^{1:i-1}, x^{i+1:N})\in \R^{d\times(N-1)}$,
\[
\mu^{1:N}(x^{1:N}) = \mu^{i|-i}(x^i\mid x^{-i}) \times \mu^{-i}(x^{-i})\,.
\] 
    We know that
    \begin{align*}
        \mc E^N(\mu^{1:N}) - N \mc E(\pi) = \mc F^N(\mu^{1:N}) - N \mc F(\pi) + \frac{\sigma^2}{2} \ent(\mu^{1:N}) - \frac{N \sigma^2}{2} \ent(\pi)\,.
    \end{align*}
    Furthermore, by Assumption~\ref{as:functional_convexity},
    \begin{align*}
        \mc F^N(\mu^{1:N}) - N \mc F(\pi)  \leq N \E_{x^{1:N} \sim \mu^{1:N}}\int \delta \mc F(\rho_{x^{1:N}}, z)\,(\rho_{x^{1:N}}(\D z) - \pi(\D z))\,.
    \end{align*}

    Using the subadditivity of entropy, we can therefore write
    \begin{align*}
        \mc E^N(\mu^{1:N}) - N\mc E(\pi)
        \le \sum_{i=1}^N \E_{x^{1:N} \sim \mu^{1:N}}\Bigl[&\delta \mc F(\rho_{x^{1:N}}, x^i) - \int \delta \mc F(\rho_{x^{1:N}}, \cdot)\, \D \pi \\
        &\qquad{} + \frac{\sigma^2}{2} \,\bigl(\ent(\mu^{i\mid -i}(\cdot \mid x^{-i})) - \ent(\pi)\bigr)\Bigr]\,.
    \end{align*}
    To decouple the terms, we now replace each $\delta \mc F(\rho_{x^{1:N}}, \cdot)$ term with $\delta \mc F(\rho_{x^{-i}}, \cdot)$:
    \begin{align*}
        &\mc E^N(\mu^{1:N}) - N\mc E(\pi) \\
        &\qquad \le \underbrace{\sum_{i=1}^N \E_{x^{1:N} \sim \mu^{1:N}}\Bigl[\delta \mc F(\rho_{x^{-i}}, x^i) - \int \delta \mc F(\rho_{x^{-i}}, \cdot)\, \D \pi + \frac{\sigma^2}{2} \,\bigl(\ent(\mu^{i\mid -i}(\cdot \mid x^{-i})) - \ent(\pi)\bigr)\Bigr]}_{\msf A} \\
        &\qquad\quad{} + {\underbrace{\sum_{i=1}^N \E_{x^{1:N}\sim \mu^{1:N}}\Bigl[ \delta \mc F(\rho_{x^{1:N}}, x^i) - \delta \mc F(\rho_{x^{-i}}, x^i) - \int \bigl(\delta \mc F(\rho_{x^{1:N}}, \cdot) - \delta \mc F(\rho_{x^{-i}}, \cdot)\bigr)\,\D \pi \Bigr]}_{\msf B}}\,.
    \end{align*}
    We consider the two terms in turn, beginning with the first.
    
    Note that by Fubini's theorem,
    \begin{align*}
        \E_{x^{1:N} \sim \mu^{1:N}}\delta \mc F(\rho_{x^{-i}}, x^i)
        &= \E_{x^{-i} \sim \mu^{-i}} \int \delta \mc F(\rho_{x^{-i}}, \cdot) \,\D\mu^{i\mid -i}(\cdot \mid x^{-i})\,.
    \end{align*}
    In order to relate the first term $\msf A$ to a KL divergence, for each $x^{-i} \in \R^{d\times (N-1)}$ we introduce the probability measure $\tau_{x^{-i}} \in \Pac$ via
    \begin{align*}
        \tau_{x^{-i}} \propto \exp\Bigl(-\frac{2}{\sigma^2}\,\delta\mc F(\rho_{x^{-i}}, \cdot)\Bigr)\,.
    \end{align*}
    We can compute
    \begin{align*}
        &\KL\bigl(\mu^{i\mid -i}(\cdot \mid x^{-i}) \bigm\Vert \tau_{x^{-i}}\bigr) \\
        &\qquad = \int \Bigl( \frac{2}{\sigma^2}\,\delta \mc F(\rho_{x^{-i}}, \cdot) + \log\mu^{i \mid -i}(\cdot \mid x^{-i})\Bigr) \,\D \mu^{i\mid -i}(\cdot \mid x^{-i}) + \log Z(\tau_{x^{-i}})\,,
    \end{align*}
    where $Z(\tau_{x^{-i}})$ is the normalization constant for $\tau_{x^{-i}}$,
    \begin{align*}
        \log Z(\tau_{x^{-i}})
        &= \log \int \exp \Bigl(-\frac{2}{\sigma^2}\,\delta \mc F(\rho_{x^{-i}}, z) \Bigr) \, \D z\\
        &= \log \int \exp\Bigl(\frac{2}{\sigma^2}\,\bigl(\delta \mc F(\pi, z) - \delta \mc F(\rho_{x^{-i}}, z)\bigr)\Bigr) \, \pi(\D z) + \log Z(\pi) \\
        &\geq -\frac{2}{\sigma^2}\int\delta \mc F(\rho_{x^{-i}}, \cdot) \, \D \pi -\ent(\pi)\,.
    \end{align*}
    Upon taking expectations, we obtain
    \begin{align*}
        \msf A
        &\le \frac{\sigma^2}{2} \sum_{i=1}^N \E_{x^{-i} \sim \mu^{-i}} \KL\bigl(\mu^{i\mid -i}(\cdot \mid x^{-i}) \bigm\Vert \tau_{x^{-i}}\bigr)\,.
    \end{align*}
    Moreover, we can recognize that $\tau_{x^{-i}}$ is a proximal Gibbs measure. By Assumptions~\ref{as:functional_smoothness} and~\ref{as:prox_gibbs_lsi},
    \begin{align*}
        \msf A
        &\le \frac{\overline C_{\msf{LSI}}\, \sigma^2}{4} \sum_{i=1}^N \E_{x^{-i} \sim \mu^{-i}} \FI\bigl(\mu^{i\mid -i}(\cdot \mid x^{-i}) \bigm\Vert \tau_{x^{-i}}\bigr) \\
        &= \frac{\overline C_{\msf{LSI}}\, \sigma^2}{4} \sum_{i=1}^N \E_{x^{1:N} \sim \mu^{1:N}}\Bigl[\Bigl\lVert\nabla_{x^i} \log \mu^{i \mid -i}(x^i \mid x^{-i}) + \frac{2}{\sigma^2}\,\nabla_{\mc W_2} \mc F(\rho_{x^{-i}}, x^i)\Bigr\rVert^2\Bigr] \\
        &= \frac{\overline C_{\msf{LSI}}\, \sigma^2}{4} \sum_{i=1}^N \E_{x^{1:N} \sim \mu^{1:N}}\Bigl[\Bigl\lVert\nabla_{x^i} \log \mu^{1:N}(x^{1:N}) + \frac{2}{\sigma^2}\,\nabla_{\mc W_2} \mc F(\rho_{x^{-i}}, x^i)\Bigr\rVert^2\Bigr] \\
        &= \frac{\overline C_{\msf{LSI}}}{\sigma^2} \sum_{i=1}^N \E_{x^{1:N} \sim \mu^{1:N}}[\norm{\nabla_{\mc W_2} \mc F(\rho_{x^{1:N}}, x^i) - \nabla_{\mc W_2} \mc F(\rho_{x^{-i}}, x^i)}^2] \\
        &= \frac{\overline C_{\msf{LSI}}}{\sigma^2} \sum_{i=1}^N \E_{x^{1:N} \sim \mu^{1:N}}[\norm{\nabla_{\mc W_2} \mc F_0(\rho_{x^{1:N}}, x^i) - \nabla_{\mc W_2} \mc F_0(\rho_{x^{-i}}, x^i)}^2] \\
        &\le \frac{\beta^2 \overline C_{\msf{LSI}}}{\sigma^2} \sum_{i=1}^N \E_{x^{1:N} \sim \mu^{1:N}}\mc W_1^2(\rho_{x^{1:N}}, \rho_{x^{-i}})\,.
    \end{align*}
    To transport the mass from $\rho_{x^{1:N}}$ to $\rho_{x^{-i}}$, we take the transport plan which moves $\frac{1}{N\,(N-1)}$ of the mass from $x^i$ to each $x^j$, $j\ne i$. It yields
    \begin{align}\label{eq:W1_bd}
        \mc W_1(\rho_{x^{1:N}}, \rho_{x^{-i}})
        &\le \frac{1}{N\,(N-1)} \sum_{\substack{j=1 \\ j \ne i}}^N {\norm{x^i - x^j}}\,.
    \end{align}
    Hence,
    \begin{align*}
        \msf A
        &\le \frac{\beta^2 \overline C_{\msf{LSI}}}{\sigma^2 N^2\,{(N-1)}^2} \E_{x^{1:N} \sim \mu^{1:N}}\sum_{i=1}^N \Bigl(\sum_{\substack{j=1 \\ j \ne i}}^N {\norm{x^i - x^j}}\Bigr)^2 \\
        &\le \frac{\beta^2 \overline C_{\msf{LSI}}}{\sigma^2 N^2\,(N-1)} \E_{x^{1:N} \sim \mu^{1:N}} \sum_{i\ne j} {\norm{x^i - x^j}^2}
        = \frac{\beta^2 \overline C_{\msf{LSI}}}{\sigma^2 N} \E_{x^{1:2} \sim \mu^{1:2}}[\norm{x^1 - x^2}^2]\,.
    \end{align*}
    We then use the inequality
    \begin{align}
        \frac{1}{2}\E_{x^{1:2}\sim\mu^{1:2}}[\norm{x^1-x^2}^2]
        &\le 2\mc W_2^2(\mu^{1:2}, \pi^{\otimes 2}) + \E_{x^{1:2} \sim \pi^{\otimes 2}}[\norm{x^1-x^2}^2]\nonumber \\
        &\le \frac{4}{N}\,\mc W_2^2(\mu^{1:N}, \pi^{\otimes N}) + 2 \E_{x\sim \pi}[\norm{x - \E x}^2]\nonumber \\
        &\le \frac{8\overline C_{\msf{LSI}}}{N}\KL(\mu^{1:N} \mmid \pi^{\otimes N}) + 2d \overline C_{\msf{LSI}}\,,\label{eq:var_under_mu}
    \end{align}
    where we used Lemma~\ref{lem:wass_exchangeable} and the Poincar\'e inequality for $\pi$.
    Hence,
    \begin{align*}
        \msf A
        &\le \frac{2\beta^2 \overline C_{\msf{LSI}}}{\sigma^2 N}\,\Bigl( \frac{8\overline C_{\msf{LSI}}}{N}\KL(\mu^{1:N} \mmid \pi^{\otimes N}) + 2d \overline C_{\msf{LSI}}\Bigr)\,.
    \end{align*}

    Next, we turn toward term $\msf B$.
    First, define a function $\zeta_{x^{1:N}}^i: \R^d \to \R$ by
    \begin{align*}
        \zeta_{x^{1:N}}^i(y)
        \deq \delta \mc F(\rho_{x^{-i}}, y) - \delta \mc F(\rho_{x^{1:N}}, y)
        = \delta \mc F_0(\rho_{x^{-i}}, y) - \delta \mc F_0(\rho_{x^{1:N}}, y)\,.
    \end{align*}
    It is clear from Assumption~\ref{as:functional_smoothness} that this function is Lipschitz with constant $2\beta \mc W_1(\rho_{x^{1:N}}, \rho_{x^{-i}})$. Thus, we obtain using this Lipschitzness,~\eqref{eq:W1_bd}, and Young's inequality,
    \begin{align*}
        \msf B
        &= \sum_{i=1}^N \E_{x^{1:N} \sim \mu^{1:N}} \int \bigl( \zeta_{x^{1:N}}^i(x^i) - \zeta_{x^{1:N}}^i(z)\bigr) \, \pi(\D z) \\
        &\leq \sum_{i=1}^N \E_{x^{1:N} \sim \mu^{1:N}} \int \frac{2\beta}{N\,(N-1)} \sum_{\substack{j=1 \\ j \neq i}}^N {\norm{x^j - x^i}\, \norm{x^i-z}} \, \pi(\D z) \\
        &\leq \frac{\beta}{N\,(N-1)} \E_{x^{1:N} \sim \mu^{1:N}}\sum_{i\ne j} {\norm{x^i - x^j}^2} + \frac{\beta}{N} \sum_{i=1}^N \E_{(x^i,z)\sim \mu^1\otimes \pi}[\norm{x^i -z}^2] \\
        &= \beta \E_{x^{1:2} \sim \mu^{1:2}}[{\norm{x^1 - x^2}^2}] + \beta \E_{(x,z)\sim \mu^1\otimes \pi}[\norm{x -z}^2]\,.
    \end{align*}
    For the first term, we can apply~\eqref{eq:var_under_mu}, and for the second term, we can apply~\eqref{eq:second_moment_bd}.
    It yields
    \begin{align*}
        \msf B
        &\le \frac{20\beta \overline C_{\msf{LSI}}}{N}\KL(\mu^{1:N} \mmid \pi^{\otimes N}) + 8\beta \overline C_{\msf{LSI}} d\,.
    \end{align*}
    Putting the bounds together with Lemma~\ref{lem:entropy_toast},
    \begin{align*}
        \KL(\mu^{1:N} \mmid \pi^{\otimes N})
        &\le \frac{33\beta \overline C_{\msf{LSI}} d}{\sigma^2}
    \end{align*}
    for all $N \ge 160\beta \overline C_{\msf{LSI}}/\sigma^2$.
    The result for $k\le N$ follows from Lemma~\ref{lem:information_ineq}.
\end{proofof}

\section{Isoperimetric Results for the Stationary Distributions}
\label{app:iso_inv_measure}

\subsection{Convexity and Smoothness}

Here, we verify the convexity and smoothness properties of $\mu^{1:N}$ in the pairwise McKean{--}Vlasov setting.

\medskip{}

\begin{proofof}{Lemma~\ref{lem:strongCvxCase}}
    For $x^{1:N} = [x^1,\dotsc,x^N] \in \R^{d\times N}$, the Hessian of $\log(1/\mu^{1:N})$ can be explicitly computed as
    \begin{align*}
        &-\frac{\sigma^2}{2}\, \nabla^2 \log \mu^{1:N}(x^{1:N}) =\begin{bmatrix}
            \nabla^2 V(x^{1}) & 0 & \cdots & 0 \\
            0 & \nabla^2 V(x^{2}) & \cdots & 0 \\
            \vdots & \vdots & \ddots & \vdots \\
            0 & 0 & \cdots & \nabla^2 V(x^N)
        \end{bmatrix} \\
        &\quad + \frac{1}{N-1}\,
        \underbrace{\begin{bmatrix}
            \sum_{j=2}^N \nabla^2 W(x^{1}- x^j) & -\nabla^2 W(x^1 - x^2) & \cdots & -\nabla^2 W(x^{1} - x^N) \\
            -\nabla^2 W(x^2 - x^1) & \sum_{\substack{j=1 \\ j \neq 2}}^N \nabla^2 W(x^2 - x^j) & \cdots & -\nabla^2 W(x^{2} - x^N) \\
            \vdots & \vdots & \ddots & \vdots \\
            -\nabla^2 W(x^N - x^1) & -\nabla^2 W(x^N - x^2) & \cdots & \sum_{j=1}^{N-1} \nabla^2 W(x^N - x^j)
        \end{bmatrix}}_{=\msf B}\,.
    \end{align*}
    Clearly, the first block matrix has eigenvalues between $\alpha_V$ and $\beta_V$.
    For the second block matrix $\msf B$, let us denote $A_{i,j}  \deq  \nabla^2 W(x^i - x^j)$ for $i,j\in [N]$. Note that $A_{i,j} = A_{j,i}$ since $W$ is even, and each $A_{i,j}$ is clearly symmetric. 
    
    For $\msf B \in \R^{dN \times dN}$ the second matrix and $y = [y^1,\dotsc,y^N]\in \R^{dN}$, we have 
    \begin{align*}
        y^\T \msf B y 
        &= \sum_{i\leq N} y_i^\T \,\Bigl(\sum_{j\in[N]\backslash i} A_{i,j}\Bigr)\, y_i
        -\sum_{\substack{i,j\leq N\\i\neq j}} y_i^\T  A_{i,j} y_j\\
        &=\sum_{\substack{i,j\leq N\\i < j}} \bigl(y_i^\T A_{i,j}y_i +y_j^\T A_{j,i}y_j - y_i^\T A_{i,j} y_j - y_j^\T A_{j,i}y_i\bigr)
        =\sum_{\substack{i,j\leq N\\i < j}} (y_i - y_j)^\T A_{i,j} (y_i - y_j)\,.
    \end{align*}
    Using $\alpha_W I_d \preceq A_{i,j} \preceq \beta_W I_d$ and
    \begin{align*}
        \msf M  \deq  \nabla^2_y \sum_{\substack{i,j\leq N\\i < j}} \norm{y_i - y_j}^2
        &=  2\,\begin{bmatrix}
            N-1  & -1 & \cdots & -1 \\
            -1 & N-1 & \cdots &  -1 \\
            \vdots & \vdots & \ddots & \vdots \\
            -1 & -1 & \cdots & N-1
        \end{bmatrix} \otimes I_{d}\,,
    \end{align*}
    we have $\frac 1 2\,\alpha_W \msf M \preceq \msf B \preceq \frac 1 2\,\beta_W \msf M$.
    Since the circulant matrix in $\msf M$ is PSD due to diagonal dominance and its largest eigenvalue is at most $N$, it follows that the eigenvalues of $\msf M$ lie between $0$ and $2N$.
    Hence, the eigenvalues of $\msf B$ lie in the interval $[\frac{N}{N-1}\,{\alpha_W^-}, \frac{N}{N-1}\,\beta_W]$.
\end{proofof}

\subsection{Bounded Perturbations}

In this section, we prove the isoperimetric results from \S\ref{sec:mkv_isoperimetry}. 
We again introduce the conditional measure: if $x^{-i} \deq  (x^{1:i-1}, x^{i+1:N})\in \R^{d\times(N-1)}$ we define 
\[
\mu^{1:N}(x^{1:N}) = \mu^{i|-i}(x^i\mid x^{-i}) \times \mu^{-i}(x^{-i})
\] 
for the conditional distribution of the $i$-th particle and the distribution of an $N$-particle system with the $i$-th particle marginalized out.

Before proceeding to the proof of Lemma~\ref{lem:bdd_perturb_meanfield}, we first state a result on log-Sobolev inequalities under weak interactions due to~\cite{otto2007new}.

\begin{lemma}[{\cite[Theorem 1]{otto2007new}}]\label{lem:otto-rez}
    Consider a measure $\mu^{1:N}$ on $\R^{d \times N}$, with conditional measures $\mu^{i|-i}$. Assume that
    \begin{align*}
        C_{\msf{LSI}}(\mu^{i|-i}(\cdot\mid x^{-i})) &\leq \frac{1}{\tau_i}\,, \qquad \text{for all $i \in [N]$, $x^{-i} \in \R^{d \times (N-1)}$}\,, \\
        \norm{\nabla_{x^i} \nabla_{x^j} \log \mu^{1:N}(x^{1:N})} &\leq \beta_{i,j}\,, \qquad \text{for all $x^{1:N} \in \R^{d \times N}$, $i, j \in [N]$, $i \neq j$}\,.
    \end{align*}
    Then, consider the matrix $A \in \R^{N \times N}$ with entries $A_{i,i} = \tau_i$, $A_{i,j} = -\beta_{i,j}$ for $i \neq j$. If $A \succeq \rho I_{N}$, then $\mu^{1:N}$ satisfies~\eqref{eq:lsi} with constant $1/\rho$.
\end{lemma}

\medskip{}

\begin{proofof}{Lemma~\ref{lem:bdd_perturb_meanfield}}
    We begin by proving the statement about $\pi$. The potential of the invariant measure $\pi$ can also be written as
    \begin{align*}
        \log\frac{1}{\pi(x)}
        &= \frac{2}{\sigma^2}\, \Bigl(V_0(x) + V_1(x) + \int \bigl(W_0(x-\cdot) + W_1(x-\cdot)\bigr) \,\D \pi \Bigr) \\
        &= \frac{2}{\sigma^2}\, \Bigl(V_0(x) + \int W_0(x-\cdot) \, \D \pi\Bigr) + \frac{2}{\sigma^2}\, \Bigl(V_1(x) + \int W_1(x-\cdot)\, \D \pi\Bigr)\,.
    \end{align*}
    This is the sum of a $\frac{2}{\sigma^2}\, (\alpha_{V_0} + \alpha_{W_0})$-convex function with a $\frac{2}{\sigma^2}\, (\osc(V_1) + \osc(W_1))$-bounded perturbation. Thus, $\pi$ satisfies \eqref{eq:lsi} with the claimed parameter.

We now prove the statement about $\mu^{1:N}$. 
Each conditional measure has a density of the form
\begin{align*}
    \mu^{i|-i}(\cdot\mid x^{-i}) \propto 
    \exp \Bigl(-\frac{2}{\sigma^2}\, V(\cdot) - \frac{2}{\sigma^2\, (N-1)} \sum_{j\in [N]\backslash i} W(\cdot - x^j)\Bigr)\,,
\end{align*}
where both $V$ and $W$ are bounded perturbations of $\alpha_{V_0}$, $\alpha_{W_0}$-strongly convex functions respectively, irrespective of the conditional variables.
Thus, by Holley--Stroock perturbation and the Bakry--\'Emery condition, each $\mu^{i | -i}(\cdot \mid x^{-i})$ satisfies \eqref{eq:lsi} with parameter
\begin{align*}
    \tau_i^{-1} \le \tau^{-1} \deq \frac{\sigma^2}{2}\,\bigl(\alpha_{V_0} + \frac{N}{N-1}\,{\alpha_{W_0}^-}\bigr)^{-1} \exp\Bigl(\frac{2}{\sigma^2}\, \bigl(\osc(V_1) + \osc(V_1)\bigr)\Bigr)\,.
\end{align*}
Secondly, we note that from~\eqref{eq:Nparticle_stat_interactive}, we have
\begin{align*}
    \frac{2}{\sigma^2\,(N-1)} \sup_{z \in \R^d} \norm{\nabla^2 W(z)}_{\msf{op}} \leq \frac{2\beta_W}{\sigma^2\, (N-1)} \eqqcolon \beta_{i,j}\,.
\end{align*}
Thus, we have
\begin{align*}
    A = \begin{bmatrix}
            \tau & -\frac{2\beta_W}{\sigma^2\, (N-1)} & \cdots & -\frac{2\beta_W}{\sigma^2\, (N-1)} \\[0.5em]
            -\frac{2\beta_W}{\sigma^2\, (N-1)} & \tau & \cdots & -\frac{2\beta_W}{\sigma^2\, (N-1)} \\[0.5em]
            \vdots & \vdots & \ddots & \vdots \\[0.5em]
            -\frac{2\beta_W}{\sigma^2\, (N-1)} & -\frac{2\beta_W}{\sigma^2\, (N-1)} & \cdots & \tau
        \end{bmatrix}\,.
\end{align*}
Under Assumption~\ref{as:bdd_perturb_VW}, this matrix is strictly diagonally dominant and has a minimum eigenvalue of at least $\tau/2$. We can now apply Lemma~\ref{lem:otto-rez} to complete the proof.
\end{proofof}

\subsection{Logarithmic Sobolev Inequalities via Perturbations}\label{scn:functional_lsi}

In this section, we state log-Sobolev inequalities for Lipschitz perturbations of strongly log-concave measures, which is used for the general McKean{--}Vlasov setting in \S\ref{sec:mfl}.

\begin{lemma}[{LSI under Lipschitz Perturbations~\cite[Theorem 1.4]{BriPed24HeatFlow}}]\label{lem:lip_perturb_lsi}

    Let $\mu \propto \exp(-H-V)$ for an $\alpha_V$-strongly convex function $V : \R^d\to\R$ and an $L$-Lipschitz function $H: \R^d \to \R$. Then, $\mu$ satisfies \eqref{eq:lsi} with constant $C_{\msf{LSI}}(\mu)$ given by
    \begin{align*}
        C_{\msf{LSI}}(\mu) \leq \frac{1}{\alpha}\exp\Bigl(\frac{L^2}{\alpha} + \frac{4L}{\sqrt\alpha}\Bigr)\,.
    \end{align*}
\end{lemma}

From this one derives the following log-Sobolev inequality for the proximal Gibbs measure.

\begin{lemma}[{Uniform LSI for the Proximal Gibbs Measure~\cite[Theorem 1]{suzuki2023convergence}}]\label{lem:uniform_lsi}

    For the proximal Gibbs measure~\eqref{eq:Nparticle_stat} in the setting of \S\ref{sec:mfl}, under Assumption~\ref{as:functional_bdd_grad}, 
    we have that
    \begin{align*}
        \sup_{\mu \in \mc{P}_2(\R^d)} C_{\msf{LSI}}(\pi_\mu) \leq \overline{C}_{\msf{LSI}}\,,
    \end{align*}
    where $\alpha$ can be bounded by
    \begin{align*}
        \overline{C}_{\msf{LSI}}
        &\le \frac{\sigma^2}{2\lambda} \exp\Bigl(\frac{2B^2}{\lambda \sigma^2} + \frac{8B}{\sqrt{2\lambda}\,\sigma}\Bigr)\,.
    \end{align*}
\end{lemma}

Obtaining a uniform-in-$N$ LSI for $\mu^{1:N}$ under Assumption~\ref{as:functional_bdd_grad} is more difficult, and we rely on the recent heat flow estimates of~\cite{BriPed24HeatFlow}.
In their work, the authors showed the existence of an $L$-Lipschitz transport map---the Kim--Milman map~\cite{KimMil12ReverseHeat}---from the standard Gaussian measure $\gamma$ to a measure $\mu$, under suitable assumptions on $\mu$.
By~\cite[Proposition 5.4.3]{bakry2014analysis}, this immediately implies that $\mu$ satisfies $C_{\msf{LSI}}(\mu) \le L^2$.
The existence of the Lipschitz transport map is based on estimates along the heat flow, and we summarize the computation in a convenient form based on bounding the operator norm of the covariance matrix of \emph{Gaussian tilts} of the measure.
The latter property is sometimes called \emph{tilt stability} in the literature.
Note that we do not attempt to optimize constants here.

\begin{lemma}[Lipschitz Transport Maps via Reverse OU]\label{lem:lip_transport}
    Let $\mu$ be a probability measure over $\R^d$ and for any $t > 0$, $y\in\R^d$, let $\mu_{t,y}$ denote the Gaussian tilt,
    \begin{align}\label{eq:gaussian_tilt}
        \mu_{t,y}(\D x) \propto \exp\Bigl(-\frac{\norm{y-x}^2}{2t} + \frac{\norm x^2}{2}\Bigr)\,\mu(\D x)\,,
    \end{align}
    where we assume that this defines a valid probability measure for all $t > 0$ and $y\in\R^d$.
    Suppose there exist $a, C > 0$ such that the following ``tilt stability'' property holds:
    \begin{align*}
        \norm{\cov_{\mu_{t,y}}}_{\msf{op}}
        &\le \Bigl(\frac{1}{\sqrt{a+1/t}} + \frac{C}{a+1/t} \Bigr)^2\,, \qquad\text{for all}~t > 0\,,\; y \in\R^d\,.
    \end{align*}
    Then, there exists an $L$-Lipschitz transport map $T : \R^d\to\R^d$ such that $T_\# \gamma = \mu$, where $\gamma$ is the standard Gaussian measure and $L$ can be estimated by
    \begin{align*}
        L
        &\le \frac{1}{\sqrt{1+a}}\exp\Bigl(\frac{C^2}{2\,(1+a)} + \frac{2C}{\sqrt{1+a}}\Bigr)\,.
    \end{align*}
\end{lemma}
\begin{proof}
    We follow the calculations of~\cite{BriPed24HeatFlow}.
    Let ${(P_t)}_{t\ge 0}$ denote the heat semigroup, and ${(Q_t)}_{t\ge 0}$ the Ornstein--Uhlenbeck semigroup.
    Then, if $\gamma$ denotes the standard Gaussian measure, the identity $Q_t f = P_{1-\exp(-2t)} f(\exp(-t)\,\cdot)$ and
    \begin{align*}
        -\frac{I}{\exp(2t)-1}
        &\preceq \nabla^2 \log Q_t\bigl(\frac{\mu}{\gamma}\bigr)
        = \exp(-2t) \,\Bigl[\nabla^2 \log P_{1-\exp(-2t)}\bigl(\frac{\mu}{\gamma}\bigr)\Bigr](\exp(-t)\,\cdot) \\
        &= \frac{1}{\exp(2t)-1}\,\Bigl( \frac{\cov_{\mu_{1-\exp(-2t),\,\exp(-t)\,\cdot}}}{1-\exp(-2t)} - I \Bigr) \\
        &\preceq \frac{I}{\exp(2t)-1} \,\Bigl[\Bigl(\frac{1/\sqrt{1-\exp(-2t)}}{\sqrt{a+1/(1-\exp(-2t))}} + \frac{C/\sqrt{1-\exp(-2t)}}{a+1/(1-\exp(-2t))}\Bigr)^2 - 1\Bigr] \\
        &\preceq \Bigl[\frac{1-\alpha}{\alpha\,(\exp(2t)-1)+1} + \frac{C^2 \exp(2t)}{(\alpha\,(\exp(2t)-1)+1)^2} \\
        &\qquad\qquad\qquad\qquad\qquad{} + \frac{2C\exp(2t)}{(\exp(2t)-1)^{1/2} \, (\alpha\,(\exp(2t)-1)+1)^{3/2}}\Bigr]\,I
    \end{align*}
    where $a = \alpha - 1$.
    Note that we can identify this with the bound of~\cite[Corollary 3.2]{BriPed24HeatFlow} with $C = L$.
    Thus, by following the calculations in the proof of~\cite[Theorem 1.4]{BriPed24HeatFlow}, we obtain the desired result.
\end{proof}

We next verify the tilt stability condition, leveraging the propagation of chaos result in Theorem~\ref{thm:entropic_prop_chaos_general_functional}.

\begin{lemma}[Tilt Stability]\label{lem:tilt_stab}
    For any $t > 0$ and $y^{1:N} \in \R^{d\times N}$, let $\mu_{t,y^{1:N}}$ be the Gaussian tilt of $\mu^{1:N}$ as defined in~\eqref{eq:gaussian_tilt} with $\mu^{1:N}$ in place of $\mu$.
    Here, $\mu^{1:N}$ is the stationary distribution~\eqref{eq:mean_field_stat_interactive} of the mean-field Langevin diffusion.
    Then, under Assumptions~\ref{as:functional_convexity},~\ref{as:functional_smoothness}, and~\ref{as:functional_bdd_grad}, and assuming that $2\lambda > \sigma^2$ and
    \begin{align}\label{eq:condition_on_N}
        N \gtrsim \frac{\beta d}{2\lambda-\sigma^2} \exp\Bigl(\frac{8B^2/\sigma^4}{2\lambda/\sigma^2-1}\Bigr)
    \end{align}
    for a sufficiently large implied (but universal) constant, we have
    \begin{align*}
        \norm{\cov_{\mu_{t,y^{1:N}}}}_{\msf{op}}
        &\le \Bigl[ \frac{1}{\sqrt{2\lambda/\sigma^2-1+1/t}} + \frac{1}{2\lambda/\sigma^2-1+1/t} \,\mc{O}\Bigl(\frac{B}{\sigma^2} + \frac{\sqrt{\beta d}}{\sigma} \exp\frac{8B^2/\sigma^4}{2\lambda/\sigma^2-1}\Bigr)\Bigr]^2\,.
    \end{align*}
\end{lemma}
\begin{proof}
    We introduce the auxiliary measures
    \begin{align}\label{eq:aux_measures}
        \pi_{t,y^i}(x^i) \propto \exp\Bigl(-\frac{\norm{y^i - x^i}^2}{2t} - \bigl(\frac{\lambda}{\sigma^2} - \frac{1}{2}\bigr)\, \norm{x^i}^2 - \frac{2}{\sigma^2}\,\delta \mc F_0(\bar \pi_{t,y^{1:N}}, x^i)\Bigr)\,, \quad i\in [N]\,,
    \end{align}
    where $\bar\pi_{t,y^{1:N}} \deq \frac{1}{N} \sum_{i=1}^N \pi_{t,y^i}$.
    To see that these auxiliary measures are well-defined, note that any minimizer of the functional
    \begin{align*}
        &(\pi^1,\dotsc,\pi^N) \\
        &\qquad{} \mapsto \sum_{i=1}^N \int \Bigl[\frac{\sigma^2\,\norm{y^i - x^i}^2}{4t} - \bigl(\frac{\lambda}{2} - \frac{\sigma^2}{4}\bigr)\, \norm{x^i}^2\Bigr] \, \pi^i(\D x^i) + N\mc F_0\bigl(\frac{1}{N}\sum_{i=1}^N \pi^i\bigr) + \frac{\sigma^2}{2} \sum_{i=1}^N \ent(\pi^i)
    \end{align*}
    satisfies the system of equations~\eqref{eq:aux_measures}, and that the minimizer is unique because the functional is strictly convex.
    We also let $\pi_{t,y^{1:N}} \deq \bigotimes_{i=1}^N \pi_{t,y^i}$.

    Then, for any unit vector $\theta^{1:N} \in \R^{d\times N}$,
    \begin{align*}
        \langle \theta^{1:N}, \cov_{\mu_{t,y^{1:N}}} \theta^{1:N}\rangle
        &\le \E_{x^{1:N}\sim \mu_{t,y^{1:N}}}[\langle \theta^{1:N}, x^{1:N} - \E_{\bar x^{1:N}\sim \pi_{t,y^{1:N}}} \bar x^{1:N} \rangle^2] \\
        &\le \Bigl(\mc W_2(\mu_{t,y^{1:N}}, \pi_{t,y^{1:N}}) + \sqrt{\langle \theta^{1:N}, \cov_{\pi_{t,y^{1:N}}} \theta^{1:N}\rangle} \Bigr)^2 \\
        &\le \Bigl(\mc W_2(\mu_{t,y^{1:N}}, \pi_{t,y^{1:N}}) + \max_{i\in [N]} \sqrt{\norm{\cov_{\pi_{t,y^i}}}_{\msf{op}}} \Bigr)^2\,,
    \end{align*}
    where we used the fact that $\pi_{t,y^{1:N}}$ is a product measure.
    Also, introduce
    \begin{align*}
        \breve \pi_{t,y^i}(x^i) \propto \exp\Bigl( - \frac{\norm{y^i - x^i}^2}{2t} -\bigl(\frac{\lambda}{\sigma^2}-\frac{1}{2}\bigr)\, \norm{x^i}^2\Bigr)\,,
    \end{align*}
    so that $\pi_{t,y^i} \propto \exp(-\frac{2}{\sigma^2}\,\delta \mc F_0(\bar \pi_{t,y^{1:N}},\cdot)) \, \breve \pi_{t,y^i}$.
    The same argument as above then yields
    \begin{align*}
        \langle \theta^i, \cov_{\pi_{t,y^i}} \theta^i\rangle
        \le \Bigl(\mc W_2(\pi_{t,y^i}, \breve\pi_{t,y^i}) + \sqrt{\langle \theta^i, \cov_{\breve \pi_{t,y^i}} \theta^i\rangle} \Bigr)^2\,.
    \end{align*}
    Since $\breve\pi_{t,y^i}$ is $(2\lambda/\sigma^2 - 1 + 1/t)$-strongly log-concave, $\norm{\cov_{\breve \pi_{t,y^i}}}_{\msf{op}} \le 1/(2\lambda/\sigma^2 - 1 + 1/t)$ by the Brascamp--Lieb inequality~\cite{BraLie1976}.
    Also, since $\frac{2}{\sigma^2}\,\delta \mc F_0(\bar\pi_{t,y^{1:N}},\cdot)$ is $2B/\sigma^2$-Lipschitz under Assumption~\ref{as:functional_bdd_grad}, then~\cite[Corollary 2.4]{KhuMaaPed24LInfOT} yields
    \begin{align*}
        \mc W_2(\pi_{t,y^i}, \breve \pi_{t,y^i})
        &\le \mc W_\infty(\pi_{t,y^i}, \breve \pi_{t,y^i})
        \le \frac{2B/\sigma^2}{2\lambda/\sigma^2 -1+1/t}\,.
    \end{align*}
    Hence,
    \begin{align*}
        \norm{\cov_{\pi_{t,y^i}}}_{\msf{op}}
        &\le \Bigl( \frac{1}{\sqrt{2\lambda/\sigma^2-1+1/t}} + \frac{2B/\sigma^2}{2\lambda/\sigma^2 -1 + 1/t}\Bigr)^2\,.
    \end{align*}
    
    Finally, it remains to control $\mc W_2(\mu_{t,y^{1:N}}, \pi_{t,y^{1:N}})$.
    Note that when $t=\infty$, this essentially reduces to Theorem~\ref{thm:entropic_prop_chaos_general_functional}, so the task is to prove a generalization thereof.
    Note that by Lemma~\ref{lem:lip_perturb_lsi}, in Theorem~\ref{thm:generalized_poc} below, we can take
    \begin{align*}
        \overline C_{\msf{LSI}}
        &\le \frac{1}{2\lambda/\sigma^2-1+1/t} \exp\Bigl(\frac{4B^2/\sigma^4}{2\lambda/\sigma^2-1+1/t} + \frac{8B/\sigma^2}{\sqrt{2\lambda/\sigma^2-1+1/t}}\Bigr) \\
        &\le \frac{3}{2\lambda/\sigma^2-1+1/t} \exp\Bigl(\frac{8B^2/\sigma^4}{2\lambda/\sigma^2-1+1/t}\Bigr)
        \le \frac{3}{2\lambda/\sigma^2-1} \exp\Bigl(\frac{8B^2/\sigma^4}{2\lambda/\sigma^2-1}\Bigr)\,.
    \end{align*}
    In particular, under the assumption~\eqref{eq:condition_on_N} for $N$, the preconditions of Theorem~\ref{thm:generalized_poc} are met and it yields the bound
    \begin{align*}
        \mc W_2(\mu_{t,y^{1:N}}, \pi_{t,y^{1:N}})
        &\lesssim \frac{\overline C_{\msf{LSI}}\sqrt{\beta d}}{\sigma}
        \lesssim \frac{\sqrt{\beta d}}{\sigma\,(2\lambda/\sigma^2-1+1/t)} \exp\Bigl(\frac{8B^2/\sigma^4}{2\lambda/\sigma^2-1}\Bigr)\,.
    \end{align*}
    Putting everything together completes the proof.
\end{proof}

\begin{corollary}[Uniform LSI for the Stationary Measure]\label{cor:unif_LSI_crazy}
    Under Assumptions~\ref{as:functional_convexity},~\ref{as:functional_smoothness}, and~\ref{as:functional_bdd_grad}, if
    \begin{align}\label{eq:second_cond_on_N}
        N \gtrsim \frac{\beta d}{\lambda} \exp\Bigl(\frac{8B^2}{\lambda \sigma^2}\Bigr)
    \end{align}
    for a sufficiently large implied (but universal) constant, then
    \begin{align*}
        C_{\msf{LSI}}(\mu^{1:N})
        &\lesssim \frac{\sigma^2}{\lambda}\exp\Bigl(\mc O\Bigl( \frac{B^2}{\lambda \sigma^2} + \frac{\beta d}{\lambda} \exp \frac{16B^2}{\lambda \sigma^2} \Bigr)\Bigr)\,.
    \end{align*}
\end{corollary}
\begin{proof}
    To meet the conditions of Lemma~\ref{lem:tilt_stab}, we perform a rescaling.
    We abuse notation and denote by $\eta : \R^d\to\R^d$ the scaling map $x \mapsto \eta x$.
    Then,
    \begin{align*}
        \mu_\eta^{1:N}(x^{1:N})
        &\deq \eta_\# \mu^{1:N}(x^{1:N})
        \propto \mu^{1:N}(\eta^{-1} x^{1:N})
        \propto \exp\Bigl(-\frac{2N}{\sigma^2}\,\mc F_0^\eta(\rho_{x^{1:N}}) - \frac{\lambda}{\eta^2 \sigma^2}\,\norm{x^{1:N}}^2\Bigr)\,,
    \end{align*}
    where $\mc F_0^\eta(\nu) \deq \mc F_0((\eta^{-1})_\# \nu)$.
    We see that $\mu_\eta^{1:N}$ is also the stationary measure for mean-field Langevin dynamics, with the following new parameters: $\beta \gets \beta/\eta^2$; $\lambda \gets \lambda/\eta^2$; $B \gets B/\eta$.
    In particular, if we take $\eta^2 = \lambda/\sigma^2$, then Lemma~\ref{lem:tilt_stab} applies to $\mu_\eta^{1:N}$ provided that~\eqref{eq:second_cond_on_N} holds.
    Together with Lemma~\ref{lem:lip_transport} with $a=1$ and $C \lesssim B/(\lambda^{1/2}\sigma) + {(\beta d/\lambda)}^{1/2} \exp(8B^2/(\lambda \sigma^2))$, it implies
    \begin{align*}
        C_{\msf{LSI}}(\mu^{1:N}_\eta)
        &\lesssim \exp\bigl(\mc O(C^2)\bigr)
        = \exp\Bigl(\mc O\Bigl( \frac{B^2}{\lambda \sigma^2} + \frac{\beta d}{\lambda} \exp \frac{16B^2}{\lambda \sigma^2} \Bigr)\Bigr)\,.
    \end{align*}
    The result for $\mu^{1:N}$ follows from contraction mapping~\cite[Proposition 5.4.3]{bakry2014analysis}.
\end{proof}

It remains to prove the following generalized propagation of chaos result.

\begin{theorem}[Generalized Propagation of Chaos]\label{thm:generalized_poc}
    For each $i\in [N]$, let $V_i : \R^d\to\R$ and let $\mc F_0 : \mc P_2(\R^d)\to \R$.
    Define probability measures
    \begin{align*}
        \mu^{1:N}(x^{1:N})
        &\propto \exp\Bigl(-\frac{2}{\sigma^2} \sum_{i=1}^N V_i(x^i) -\frac{2N}{\sigma^2}\,\mc F_0(\rho_{x^{1:N}})\Bigr)\,, \\
        \pi^i(x^i)
        &\propto \exp\Bigl(-\frac{2}{\sigma^2} \,V_i(x^i) - \frac{2}{\sigma^2} \,\delta\mc F_0(\bar \pi, x^i)\Bigr)\,,
    \end{align*}
    where $\bar \pi \deq \frac{1}{N}\sum_{i=1}^N \pi_i$, $\pi^{1:N} \deq \bigotimes_{i=1}^N \pi^i$, and we assume that these measures are all well-defined.
    Furthermore, assume that $\mc F_0$ satisfies Assumptions~\ref{as:functional_convexity} and~\ref{as:functional_smoothness}, and that for all $i \in [N]$ and all $\nu \in \mc P_2(\R^d)$, the proximal Gibbs measure
    \begin{align*}
        \pi_\nu^i \propto \exp\Bigl(-\frac{2}{\sigma^2}\,\bigl(V_i + \delta \mc F_0(\nu,\cdot)\bigr)\Bigr)
    \end{align*}
    satisfies~\eqref{eq:lsi} with constant $\overline C_{\msf{LSI}}$.
    Then, for all $N\gtrsim \beta \overline C_{\msf{LSI}} d/\sigma^2$,
    \begin{align*}
        \frac{1}{2\overline C_{\msf{LSI}}} \,\mc W_2^2(\mu^{1:N}, \pi^{1:N}) \le \KL(\mu^{1:N} \mmid \pi^{1:N})
        \lesssim \frac{\beta \overline C_{\msf{LSI}} d}{\sigma^2}\,.
    \end{align*}
\end{theorem}
\begin{proof}
    Since the proof closely follows the proof of Theorem~\ref{thm:entropic_prop_chaos_general_functional}, to avoid repetition we only highlight the main changes.
    We define the following energy functionals:
    \begin{align*}
        \mc E^N(\nu^{1:N})
        &\deq \sum_{i=1}^N \int V_i \, \D \nu^i + N\int \mc F_0(\rho_{x^{1:N}})\, \nu^{1:N}(\D x^{1:N}) + \frac{\sigma^2}{2} \ent(\nu^{1:N})\,, \\
        \mc E(\nu^{1:N})
        &\deq \sum_{i=1}^N \int V_i \, \D \nu^i + N\mc F_0(\bar \nu) + \frac{\sigma^2}{2} \ent(\nu^{1:N})\,,
    \end{align*}
    where for a measure $\nu^{1:N}$ on $\R^{d\times N}$, we use the notation $\bar\nu \deq \frac{1}{N} \sum_{i=1}^N \nu^i$ for the average of the marginals.
    The first step is to establish the analogue of the entropy toast inequality (Lemma~\ref{lem:entropy_toast}).
    Letting $Z \deq \prod_{i=1}^N \int \exp(-\frac{2}{\sigma^2}\,V_i - \frac{2}{\sigma^2}\,\delta \mc F_0(\bar \pi,\cdot))$ denote the normalizing constant for $\pi^{1:N}$,
    \begin{align*}
        \KL(\nu^{1:N} \mmid \pi^{1:N})
        &= \frac{2}{\sigma^2}\sum_{i=1}^N \int V_i \, \D \nu^i + \frac{2N}{\sigma^2} \iint \delta \mc F_0(\bar \pi, z) \,\rho_{x^{1:N}}(\D z) \,\nu^{1:N}(\D x^{1:N}) + \ent(\nu^{1:N}) \\
        &\qquad{} + \log Z \\
        &= \frac{2}{\sigma^2}\sum_{i=1}^N \int V_i \, \D (\nu^i - \pi^i) + \frac{2N}{\sigma^2} \E_{x^{1:N} \sim \nu^{1:N}} \int \delta \mc F_0(\bar \pi, \cdot) \, \D(\rho_{x^{1:N}}- \bar\pi) \\
        &\qquad{} + \ent(\nu^{1:N}) - \ent(\pi^{1:N}) \\
        &\leq \frac{2}{\sigma^2}\sum_{i=1}^N \int V_i \, \D (\nu^i - \pi^i) + \frac{2N}{\sigma^2} \E_{x^{1:N} \sim \nu^{1:N}}[\mc F_0(\rho_{x^{1:N}}) - \mc F_0(\bar \pi)] \\
        &\qquad{} + \ent(\nu^{1:N}) - \ent(\pi^{1:N}) \\
        &\leq \frac{2}{\sigma^2}\, \bigl(\mc E^N(\nu^{1:N}) - \mc E(\pi^{1:N})\bigr)\,.
    \end{align*}
    From here, we find that
    \begin{align*}
         \E_{x^{1:N} \sim \nu^{1:N}}[\mc F_0(\rho_{x^{1:N}}) - \mc F_0(\bar \pi)]
         \leq  \E_{x^{1:N} \sim \nu^{1:N}}\int \delta \mc F_0(\rho_{x^{1:N}},\cdot) \,\D(\rho_{x^{1:N}}- \bar \pi)\,.
    \end{align*}
    
    Now, moving on to the propagation of chaos part of this argument, we have
    \begin{align*}
        &\mc E^N(\mu^{1:N}) - \mc E(\pi^{1:N}) \\
        &\qquad \leq \sum_{i=1}^N \E_{x^{1:N} \sim \mu^{1:N}}\Bigl[\int V_i \,\D( \mu^{i\mid -i}(\cdot\mid x^{-i}) - \pi^i) + \int \delta \mc F_0(\rho_{x^{1:N}}, \cdot) \, \D(\rho_{x^{1:N}} - \pi^i)\Bigr] \\
        &\qquad\qquad{} + \frac{\sigma^2}{2} \sum_{i=1}^N\E_{x^{-i} \sim \mu^{-i}}\bigl[\ent(\mu^{i\mid -i}(\cdot\mid x^{-i})) -\ent(\pi^i)\bigr] \\
        &\qquad = \sum_{i=1}^N \E_{x^{1:N} \sim \mu^{1:N}}\Bigl[\int V_i \,\D (\mu^{i\mid -i}(\cdot\mid x^{-i}) - \pi^i)  +  \delta \mc F_0(\rho_{x^{1:N}}, x^i) -\int \delta \mc F_0(\rho_{x^{1:N}}, \cdot) \, \D\pi^i\Bigr] \\
        &\qquad\qquad{} + \frac{\sigma^2}{2} \sum_{i=1}^N\E_{x^{-i} \sim \mu^{-i}}\bigl[\ent(\mu^{i\mid -i}(\cdot\mid x^{-i})) -\ent(\pi^i)\bigr]\,.
    \end{align*}
    To decouple, introduce a new variable $z^i \sim \mu^i$ independent of all the others and define as a shorthand $\tilde x_i^{1:N}$ as the vector $x^1, \ldots, x^{i-1}, z^i, x^{i+1}, \ldots, x^N$.
    \begin{align*}
        &\mc E^N(\mu^{1:N}) -\mc E(\pi^{1:N}) \\
        &\qquad \leq \sum_{i=1}^N \E_{x^{1:N} \sim \mu^{1:N}}\Bigl[\int V_i \,\D (\mu^{i\mid -i}(\cdot\mid x^{-i}) - \pi^i) + \frac{\sigma^2}{2}\,\bigl(\ent(\mu^{i\mid -i}(\cdot\mid x^{-i})) -\ent(\pi^i)\bigr)\Bigr] \\
        &\qquad\qquad{} + \sum_{i=1}^N \E_{x^{1:N}\sim \mu^{1:N}}\Bigl[\E_{z^i \sim \mu^i} \delta \mc F_0(\rho_{\tilde x^{1:N}_i}, x^i) - \int \E_{z^i \sim \mu^i} \delta \mc F_0(\rho_{\tilde x^{1:N}_i}, \cdot) \, \D \pi^i\Bigr] \\
        &\qquad{} + \sum_{i=1}^N \E_{x^{1:N} \sim \mu^{1:N}}\Bigl[ \delta \mc F_0(\rho_{x^{1:N}}, x^i) - \E_{z^i \sim \mu^i} \delta \mc F_0(\rho_{\tilde x^{1:N}_i}, x^i) \\
        &\qquad\qquad\qquad\qquad\qquad{} - \int \bigl(\delta \mc F_0(\rho_{x^{1:N}}, \cdot) -\E_{z^i \sim \mu^i} \delta \mc F_0(\rho_{\tilde x^{1:N}_i}, \cdot)\bigr) \, \D \pi^i\Bigr]\,.
    \end{align*}
    We group the terms in the first two lines as $\msf{A}$, and the terms in the last two lines as $\msf B$.
    
    Let us first look at $\msf{A}$.
    If we introduce the proximal Gibbs measure for $\rho_{\tilde x^{1:N}_i}$ (in the $i$-th coordinate), defined via
    \begin{align*}
        \tau_{\tilde x^{1:N}_i}^i \propto \exp\Bigl(-\frac{2}{\sigma^2}\, \bigl(V_i + \delta \mc F_0(\rho_{\tilde{x}_i^{1:N}}, \cdot) \bigr) \Bigr)\,,
    \end{align*}
    one obtains as before
    \begin{align*}
        \msf{A} \leq \frac{\sigma^2}{2} \sum_{i=1}^N \E_{x^{1:N} \sim \mu^{1:N}} \E_{z^i \sim \mu^i} \KL\bigl(\mu^{i\mid -i}(\cdot\mid x^{-i})\bigm\Vert \tau_{\tilde x^{1:N}_i}^i\bigr)\,.
    \end{align*}
    Applying the log-Sobolev inequality, it yields
    \begin{align*}
        \msf A
        &\leq \frac{\beta^2 \overline{C}_{\msf{LSI}}}{\sigma^2} \sum_{i=1}^N \E_{x^{1:N} \sim \mu^{1:N}} \E_{z^i \sim \mu^i} \mc W_1^2(\rho_{x^{1:N}}, \rho_{\tilde x_i^{1:N}})\,.
    \end{align*}
    The Wasserstein distance is bounded by
    \begin{align*}
        \mc W_1(\rho_{x^{1:N}}, \rho_{\tilde x_i^{1:N}})
        &\le \frac{1}{N}\, {\norm{x^i - z^i}}\,.
    \end{align*}
    It eventually yields, as before,
    \begin{align*}
        \msf A
        &\lesssim \frac{\beta^2 \overline C_{\msf{LSI}}}{\sigma^2 N}\,\Bigl( \frac{\overline C_{\msf{LSI}}}{N}\KL(\mu^{1:N} \mmid \pi^{1:N}) + d\overline C_{\msf{LSI}}\Bigr)\,.
    \end{align*}
        As for the term $\msf{B}$,
        a straightforward modification of the proof of Theorem~\ref{thm:entropic_prop_chaos_general_functional} readily yields
        \begin{align*}
            \msf{B} \lesssim \frac{\beta \overline{C}_{\msf{LSI}}}{N} \KL(\mu^{1:N} \mmid \pi^{1:N}) + \beta\overline{C}_{\msf{LSI}}d\,.
        \end{align*}
        Putting everything together yields the result.
\end{proof}

\section{Explicit Calculations for the Gaussian Case}\label{sec:cal-Gaussian}

Here we provide complete details for Example~\ref{eg:Gaussian}: for any $k\leq N$,
\[
\frac{dk^2}{N^2} \lesssim \KL(\mu^{1:k} \mmid \pi^{\otimes k}) \lesssim \frac{dk^2}{N^2}\log N \,.
\]
Note that for $\msf C\in\R^{N\times N}$ with $\msf C_{i,i} = N-1$ and $\msf C_{i,j} = -1$ if $i\neq j$,
\[
\mu^{1:N} = \mc N\Bigl(0,\underbrace{\frac{\sigma^{2}}{2}\,\bigl(I_{N} \otimes A +\frac{\lambda}{N-1}\,\msf C\otimes I_{d}\bigr)^{-1}}_{\eqqcolon \Sigma_{1}}\Bigr)\qquad\text{and}\qquad\pi = \mc N\Bigl(0,\underbrace{\frac{\sigma^{2}}{2}\,{(A+\lambda I_{d})}^{-1}}_{\eqqcolon \Sigma_{2}}\Bigr)\,.
\]
The $k$-particle marginal $\mu^{1:k}$ is a Gaussian with zero mean and covariance being the upper-left $(kN\times kN)$-block matrix of $\Sigma_{1}$, which we denote by $\Sigma_{1,k}$. 
Clearly, $\pi^{\otimes k}$ is also a Gaussian with zero mean and covariance $\Sigma_{2,k} \deq I_{k}\otimes\Sigma_{2}$.
From a well-known formula for the $\KL$ divergence between two Gaussian distributions,
\begin{equation}
\KL(\mu^{1:k}\mmid\pi^{\otimes k})=\frac{1}{2}\,\bigl(-\log\det(\Sigma_{2,k}^{-1}\Sigma_{1,k})-dk+\tr(\Sigma_{2,k}^{-1}\Sigma_{1,k})\bigr)\,.\label{eq:KL-Gaussian}
\end{equation}

Let $1_{p}\in\R^{p}$ be the $p$-dimensional vector with all entries $1$. 
From $\msf C=NI_{N}-1_{N}1_{N}^{\T}$,
\begin{align*}
\frac{2}{\sigma^{2}}\,\Sigma_{1}
&=\Bigl(I_{N}\otimes\bigl(\underbrace{A+\frac{\lambda N}{N-1}\,I_{d}}_{\eqqcolon A_{\lambda}}\bigr)-\frac{\lambda}{N-1}\,(1_{N}1_{N}^{\T})\otimes I_{d}\Bigr)^{-1} \\
&\underset{\text{(i)}}{=}\Bigl(I_{N}\otimes A_{\lambda}-\frac{\lambda}{N-1}\,(1_{N}\otimes I_{d})(1_{N}^{\T}\otimes I_{d})\Bigr)^{-1}\\
 & \underset{\text{(ii)}}{=}\Par{I_{N}\otimes A_{\lambda}}^{-1} \\
 &\qquad{} -{(I_{N}\otimes A_{\lambda})}^{-1}(1_{N}\otimes I_{d}) \\
 &\qquad\qquad{}\times \bigl(I_{d}+(1_{N}^{\T}\otimes I_{d})\Par{I_{N}\otimes A_{\lambda}}^{-1}(1_{N}\otimes I_{d})\bigr)^{-1}(1_{N}^{\T}\otimes I_{d})\Par{I_{N}\otimes A_{\lambda}}^{-1}\\
 & \underset{\text{(iii)}}{=}I_{N}\otimes A_{\lambda}^{-1}- (1_N \otimes A_\lambda^{-1})\bigl(I_d + (1_N^\T 1_N) \otimes A_\lambda^{-1}\bigr)^{-1} (1_N^\T \otimes A_\lambda^{-1}) \\
 &= I_{N}\otimes A_{\lambda}^{-1}- (1_N \otimes A_\lambda^{-1}) {(I_d + NA_\lambda^{-1})}^{-1} (1_N^\T \otimes A_\lambda^{-1}) \\
 &= I_{N}\otimes A_{\lambda}^{-1}-(1_{N}1_{N}^{\T})\otimes {(A_{\lambda}^{2}+NA_{\lambda})}^{-1}\,,
\end{align*}
where in (i) we used $(A\otimes B)(C\otimes D)=(AC)\otimes(BD)$,
(ii) follows from the Woodbury matrix identity, and (iii) used $(A\otimes B)^{-1}=A^{-1}\otimes B^{-1}$.
Hence it follows that
\[
\frac{2}{\sigma^{2}}\,\Sigma_{1,k}=I_{k}\otimes A_{\lambda}^{-1}-(1_{k}1_{k}^{\T})\otimes {(A_{\lambda}^{2}+NA_{\lambda})}^{-1}\,.
\]

By the spectral decomposition of $A$, we can write $A=UDU^{\T}$ for a diagonal $D \in \R^{d\times d}$ and an orthogonal matrix $U \in \R^{d\times d}$ such that $\{\sigma_i \deq D_{i,i}\}_{i\in [d]}$ correspond to the eigenvalues of $A$. 
Since $\log\det(\cdot)$ and $\tr(\cdot)$ in (\ref{eq:KL-Gaussian}) are orthogonally invariant, let us look at the orthogonal conjugate of $\Sigma_{2,k}^{-1}\Sigma_{1,k}$ by $I_{k}\otimes U^{\T}\in\R^{dk\times dk}$. Using $(A\otimes B)^{\T}=A^{\T}\otimes B^{\T}$
and denoting $D_{\lambda} \deq D+\frac{\lambda N}{N-1}\,I_{d}$,
\begin{align*}
&(I_{k}\otimes U^{\T})\Sigma_{2,k}^{-1}\Sigma_{1,k}(I_{k}\otimes U)\\
&\qquad = (I_{k}\otimes U^{\T})\bigl(I_{k}\otimes(A+\lambda I_{d})\bigr)\bigl(I_{k}\otimes A_{\lambda}^{-1}-(1_{k}1_{k}^{\T})\otimes {(A_{\lambda}^{2}+NA_{\lambda})}^{-1}\bigr)(I_{k}\otimes U)\\
&\qquad = \bigl(I_{k}\otimes(D+\lambda I_{d})\bigr)\bigl(I_{k}\otimes D_{\lambda}^{-1}-(1_{k}1_{k}^{\T})\otimes {(D_{\lambda}^{2}+ND_{\lambda})}^{-1}\bigr)\\
&\qquad = I_{k}\otimes\Par{(D+\lambda I_{d})D_{\lambda}^{-1}}-\underbrace{(1_{k}1_{k}^{\T})}_{\eqqcolon J_{k}}\otimes\underbrace{\bigl((D+\lambda I_{d}) {(D_{\lambda}^{2}+ND_{\lambda})}^{-1}\bigr)}_{\eqqcolon S_\lambda}\\
&\qquad = I_{dk}-\Bigl(\underbrace{\frac{\lambda}{N-1}\,I_{k}\otimes D_{\lambda}^{-1}+J_{k}\otimes S_{\lambda}}_{\eqqcolon \msf M}\Bigr)\,.
\end{align*}
For $\sigma_d  \deq  \min_{i\in [d]} \sigma_i$, $\alpha \deq \sigma_d + \lambda$, and $\varepsilon \deq  \lambda/(N-1)$, 
we have $D_\lambda^{-1} \precsim \frac{1}{\alpha}\,I_d$ and $S_\lambda \precsim \frac{1}{\alpha+N}\,I_d$ due to
\[
    ((D_{\lambda})^{-1})_{i,i} \leq \frac{1}{\alpha + \varepsilon}
    \qquad\text{and}\qquad
    (S_{\lambda})_{i,i} \leq \frac{\alpha}{(\alpha + \varepsilon)\,(\alpha + \varepsilon + N)}\,.
\]
Since the eigenvalues of $A\otimes B$ consist of all possible combinations arising from the product of eigenvalues, one from $A$ and one from $B$, the largest eigenvalue $\eta_1$ of $\msf M$ is less than $1$: 
\[
\eta_1
\leq\frac{\lambda}{N-1}\,\norm{D_\lambda^{-1}}+k\,\norm{S_\lambda}
\leq \frac{\varepsilon}{\alpha + \varepsilon}+\frac{\alpha N}{(\alpha + \varepsilon)\,(\alpha + \varepsilon + N)}
= \frac{\varepsilon + N}{\alpha + \varepsilon + N}\,.
\]
Denoting the eigenvalues of $\msf M$ by $\eta_{i}$, it follows from \eqref{eq:KL-Gaussian} that
\begin{align*}
    2\KL(\mu^{1:k}\mmid\pi^{\otimes k})
    &=- \bigl(\log\det(I_{dk}-\msf M)+dk-\tr(I_{dk}-\msf M)\bigr)
    =-\sum_{i=1}^{dk} \bigl(\log(1-\eta_{i})+\eta_{i}\bigr) \\
    &= \sum_{i=1}^{dk} \sum_{n\geq 2}\frac{\eta_i^n}{n}\,.
\end{align*}
Then, we have a trivial lower bound of $\frac{1}{2}\sum_{i=1}^{dk} \eta_i^2$, and as for the upper bound,
\[
    \sum_{i=1}^{dk} \sum_{n\geq 2} \frac{\eta_i^n}{n}
    \leq \sum_{i=1}^{dk}\Bigl(\eta_{i}^{2}+\sum_{n\geq 1}\frac{\eta_i^{n+2}}{n}\Bigr)
    = \sum_{i=1}^{dk} \eta_{i}^{2} \log\bigl(\frac{e}{1-\eta_i}\bigr) 
    \lesssim \bigl(1 \vee \log\frac{N}{\alpha}\bigr) \tr(\msf M^{2})\,,
\]
where the last inequality follows from $(1-\eta_i)^{-1}\leq (1-\eta_1)^{-1}$.

Using $\tr(A\otimes B)=\tr(A)\cdot\tr(B)$, and $D_\lambda^{-1} \precsim \frac{1}{\alpha}\,I_d$ and $S_\lambda \precsim \frac{1}{\alpha+N}\,I_d$, we have 
\begin{align*}
\tr(\msf M^{2})
&\lesssim \frac{\lambda^2}{{(N-1)}^{2}}\tr(I_{k})\tr(D_{\lambda}^{-2})+\tr(J_{k}^2)\tr(S_{\lambda}^{2}) \\
 & \lesssim \frac{\lambda^2}{\alpha^2}\,\frac{dk}{N^2}+\frac{dk^{2}}{{(\alpha+N)}^{2}}
 \lesssim \frac{dk^{2}}{N^{2}}\,.
\end{align*}
As for the lower bound, since $(S_\lambda)_{i,i} \sim \frac{1}N$ for large $N$, we have
\[
\tr(\msf M^2)
\gtrsim \frac{dk^2}{N^2}\,,
\]
which completes the proof.

\section{Additional Technical Lemmas}

In our proofs, we used the following general lemmas on exchangeability.

\begin{lemma}\label{lem:wass_exchangeable}
    Let $\mu^{1:N}$, $\nu^{1:N}$ be two exchangeable measures over $\R^{d\times N}$.
    For any $k\le N$,
    \begin{align*}
        \mc W_2^2(\mu^{1:k}, \nu^{1:k})
        &\le \frac{k}{N}\,\mc W_2^2(\mu^{1:N}, \nu^{1:N})\,.
    \end{align*}
\end{lemma}
\begin{proof}
    Let $(X^{1:N}, Y^{1:N})$ be optimally coupled for $\mu^{1:N}$ and $\nu^{1:N}$.
    For each subset $S\subseteq [N]$ of size $k$, it induces a coupling $(X^S, Y^S)$ of $\mu^{1:k}$ and $\nu^{1:k}$ (by exchangeability).
    In particular, the law of $(X^{\msf S}, Y^{\msf S})$, where $\msf S$ is an independent and uniformly random subset of size $k$, is also a coupling of $\mu^{1:k}$ and $\nu^{1:k}$. Hence,
    \begin{align*}
        \mc W_2^2(\mu^{1:k}, \nu^{1:k})
        &\le \E[\norm{X^{\msf S} - Y^{\msf S}}^2]
        = \frac{1}{\binom{N}{k}} \sum_{\abs S = k} \E[\norm{X^S - Y^S}^2] \\
        &= \frac{1}{\binom{N}{k}} \sum_{i=1}^N \sum_{\abs S = k \, : \, i \in S} \E[\norm{X^i - Y^i}^2]
        = \frac{\binom{N-1}{k-1}}{\binom{N}{k}} \sum_{i=1}^N \E[\norm{X^i - Y^i}^2] \\
        &= \frac{k}{N}\,\mc W_2^2(\mu^{1:N}, \nu^{1:N})\,,
    \end{align*}
    which completes the proof.
\end{proof}

\begin{lemma}[{Information Inequality~\cite{csiszar1984sanov}}]\label{lem:information_ineq}
    If $\mc X^1, \ldots, \mc X^N$ are Polish spaces and $\mu^{1:N}$, $\nu^{1:N}$ are probability measures on $\mc X^1 \times \cdots \times \mc X^N$, where $\nu^{1:N} = \nu^1 \otimes \cdots \otimes \nu^N$ is a product measure, then for the marginals $\mu^i$ of $\mu$, it holds that
    \[
        \sum_{i=1}^N \KL(\mu^i \mmid \nu^i) \leq \KL(\mu^{1:N} \mmid \nu^{1:N})\,.
    \]
    In particular when $\mu^{1:N}$, $\nu^{1:N}$ are both exchangeable, this states that $\KL(\mu^1 \mmid \nu^1) \leq \frac{1}{N} \KL(\mu^{1:N} \mmid \nu^{1:N})$.
\end{lemma}

Note that Lemma~\ref{lem:information_ineq} follows from the chain rule and convexity of the $\KL$ divergence.

\section{Sampling Guarantees}\label{app:sampling}

Here, we show how to obtain the claimed rates in \S\ref{scn:main_results}.
We begin with some preliminary facts.

\paragraph*{$\KL$ divergence guarantees.}
To obtain our guarantees in $\KL$ divergence, we use the following lemma.

\begin{lemma}[{\cite[Proof of Theorem 6]{zhang2023improved}}]
    Let $\hat \mu$, $\mu$, and $\pi$ be three probability measures, and assume that $\mu$ satisfies~\eqref{eq:lsi} with constant $C_{\msf{LSI}}(\mu)$. Then,
    \begin{align*}
        \KL(\hat\mu \mmid \pi)
        &\le 2\,\chi^2(\hat\mu\mmid \mu) + \KL(\mu \mmid \pi) + \frac{C_{\msf{LSI}}(\mu)}{4}\FI(\mu \mmid \pi)\,.
    \end{align*}
\end{lemma}

We instantiate the lemma with $\hat\mu^{1:N}$, $\mu^{1:N}$, and $\pi^{\otimes N}$ respectively.
In the setting of Theorem~\ref{thm:weak_prop_chaos}, it is seen that $\KL(\mu^{1:N} \mmid \pi^{\otimes N})$ and $C_{\msf{LSI}}(\mu^{1:N})\FI(\mu^{1:N} \mmid \pi^{\otimes N})$ are of the same order and can be made at most $N\varepsilon^2$ if we take $N \asymp \kappa^4 d/\varepsilon^2$.
Thus, if we have a sampler that achieves $\chi^2(\hat \mu^{1:N} \mmid \mu^{1:N}) \le N\varepsilon^2$, it follows from exchangeability (Lemma~\ref{lem:information_ineq}) that $\KL(\hat\mu^1 \mmid \pi) \lesssim \varepsilon^2$.

\paragraph*{Guarantees using the sharp propagation of chaos bound.}
Here, we impose Assumptions~\ref{as:smoothness},~\ref{as:pi_lsi},~\ref{as:weak_interaction}, and~\ref{as:lsi_N}.
It follows from Theorem~\ref{thm:N_choice_bias} that $N = \widetilde\Theta(\sqrt d/\varepsilon)$ suffices in order to make $\nicefrac{\sqrt\alpha}{\sigma}\,\mc W_2(\mu^{1},\pi) \le \varepsilon$.
For the first term, we use exchangeability (Lemma~\ref{lem:wass_exchangeable}) to argue that
\begin{align*}
    \mc W_2(\hat\mu^{1}, \mu^{1}) \le N^{-1/2}\,\mc W_2(\hat\mu^{1:N},\mu^{1:N})\,,
\end{align*}
and hence we invoke sampling guarantees to ensure that $\nicefrac{\sqrt{\alpha}}{\sigma}\,\mc W_2(\hat\mu^{1:N}, \mu^{1:N}) \le N^{1/2} \varepsilon$ under~\eqref{eq:lsi}.
\begin{itemize}
    \item \textbf{LMC:} We use the guarantee for Langevin Monte Carlo from~\cite{vempala2019rapid}.
    \item \textbf{MALA--PS:} We use the guarantee for the Metropolis-adjusted Langevin algorithm together with the proximal sampler from~\cite{altschuler2023faster}. Note that the iteration complexity is $\widetilde{\mc O}(\kappa d^{1/2} N^{1/2})$, and we substitute in the chosen value for $N$.
    
    \item \textbf{ULMC--PS:} Here, we use underdamped Langevin Monte Carlo to implement the proximal sampler.
    To justify the sampling guarantee, note that since $\log\mu^{1:N}$ is $\beta$-smooth, if we choose step size $h = \frac{1}{2\beta}$ for the proximal sampler, then the RGO is $\beta$-strongly log-concave and $3\beta$-log-smooth.
    According to~\cite[Proof of Theorem 5.3]{altschuler2023faster}, it suffices to implement the RGO in each iteration to accuracy $N^{1/2} \varepsilon/\kappa^{1/2}$ in $\sqrt{\KL}$.
    Then, from~\cite{zhang2023improved}, this can be done via ULMC with complexity $\widetilde{\mc O}(\kappa^{1/2} d^{1/2}/\varepsilon)$.
    Finally, since the number of outer iterations of the proximal sampler is $\widetilde{\mc O}(\kappa)$, we obtain the claim.

    \item \textbf{ULMC${}^+$:} Here, we use either the randomized midpoint discretization~\cite{shen2019randomized} or the shifted ODE discretization~\cite{FosLyoObe21ShiftedODE} of the underdamped Langevin diffusion.
    We also replace the LSI assumptions (Assumptions~\ref{as:pi_lsi} and~\ref{as:lsi_N}) with strong convexity (Assumption~\ref{as:str_cvx_VW}).
\end{itemize}

\paragraph*{Guarantees under strong displacement convexity.}
Here, we impose Assumptions~\ref{as:smoothness} and~\ref{as:str_cvx_VW}. As discussed above, to obtain $\KL$ guarantees, we require log-concave samplers that can achieve $\chi^2(\hat\mu^{1:N} \mmid \mu^{1:N}) \le N\varepsilon^2$.
For $\mc W_2$ guarantees, by Theorem~\ref{thm:weak_prop_chaos} we take $N \asymp \kappa^2 d/\varepsilon^2$ and we require log-concave samplers that can achieve $\nicefrac{\sqrt\alpha}{\sigma}\,\mc W_2(\hat\mu^{1:N}, \mu^{1:N}) \le N^{1/2}\varepsilon$.
\begin{itemize}
    \item \textbf{LMC:} For Langevin Monte Carlo, we use the $\chi^2$ guarantee from~\cite{chewi2021analysis} and the $\mc W_2$ guarantee from~\cite{DurMajMia19LMCCvxOpt}.
    \item \textbf{ULMC:} For underdamped Langevin Monte Carlo, we use the $\chi^2$ guarantee from~\cite{altschuler2023faster}.
    \item \textbf{ULMC${}^+$:} Here, we use the $\mc W_2$ guarantees for either the randomized midpoint discretization~\cite{shen2019randomized} or the shifted ODE discretization~\cite{FosLyoObe21ShiftedODE} of the underdamped Langevin diffusion.
\end{itemize}

\paragraph*{Guarantees in the general McKean{--}Vlasov setting.}
In the setting of Theorem~\ref{thm:entropic_prop_chaos_general_functional}, we take $N\asymp \kappa d/\varepsilon^2$.
We use the same sampling guarantees under~\eqref{eq:lsi} as in the prior settings.

We also note that in order to apply the log-concave sampling guarantees, we must check that $\mu^{1:N}$ is log-smooth.
This follows from Assumption~\ref{as:functional_smoothness}.
Indeed,
\begin{align*}
    \norm{\nabla \log \mu^{1:N}(x^{1:N}) - \nabla \log \mu^{1:N}(y^{1:N})}
    &= \frac{2}{\sigma^2}\sqrt{\sum_{i=1}^N {\norm{\nabla_{\mc W_2} \mc F(\rho_{x^{1:N}}, x^i) - \nabla_{\mc W_2} \mc F(\rho_{y^{1:N}}, y^i)}^2}} \\
    &\le \frac{2\sqrt 2\,\beta}{\sigma^2}\sqrt{\sum_{i=1}^N \bigl(\norm{x^i - y^i}^2 + \mc W_1^2(\rho_{x^{1:N}}, \rho_{y^{1:N}})\bigr)} \\
    &\le \frac{4\beta}{\sigma^2}\,\norm{x^{1:N} - y^{1:N}}\,.
\end{align*}

\end{document}